\documentclass[11pt,reqno]{amsart}
\usepackage{graphicx,amsmath,amsthm,verbatim,amssymb,lineno,titletoc,pgfplots}
\pgfplotsset{compat=newest}
\usepgfplotslibrary{fillbetween}
 
\usepackage{mathrsfs}
\usepackage{adjustbox}
\usepackage{ytableau}
\usepackage{hyperref}
\usepackage{array}
\usepackage{tikz}
\usetikzlibrary{shapes.multipart,patterns,arrows}

\usepackage[font=small]{caption}
\hypersetup{colorlinks,linkcolor={red},citecolor={olive},urlcolor={red}}
\usepackage{amssymb}
\usepackage[T1]{fontenc} 
\usepackage{fourier} 
\usepackage[english]{babel} 
\usepackage{amsmath,amsfonts,amsthm,bbm} 
\usepackage{appendix} 

\usepackage[all]{xy}
\usepackage{enumerate}
\usepackage{mathrsfs}
\usepackage[english]{babel}
\usepackage{multirow}
\usepackage{float}
\usepackage{enumitem}
\usepackage{graphicx,accents}

\usepackage{mathtools}

\mathtoolsset{showonlyrefs}

\usepackage[top=1.3in, left=1.6in, right=1.6in, bottom=1.3in]{geometry}

\newcolumntype{M}[1]{>{\centering\arraybackslash}m{#1}}
\newcolumntype{N}{@{}m{0pt}@{}}

\def\liminf{\mathop{\rm lim\,inf}\limits}

\def\R{\mathbb{R}}

\def\P{\mathbb{P}}
\def\E{\mathbb{E}}
\def\1{\mathbf{1}}

\def\L{\cev{\bullet}}
\def\R{\vec{\bullet}}
\def\B{\dot{\bullet}}
\def\b{\bullet}

\DeclareRobustCommand{\cev}[1]{\reflectbox{\ensuremath{\vec{\reflectbox{\ensuremath{#1}}}}}}
\usepackage{theoremref}

\newtheorem{theorem}{Theorem}
\newtheorem{lemma}[theorem]{Lemma}
\newtheorem{prop}[theorem]{Proposition}

\newtheorem{remark}[theorem]{\textbf{\textup{Remark}}}
\theoremstyle{definition}

\DeclareRobustCommand{\cev}[1]{\reflectbox{\ensuremath{\vec{\reflectbox{\ensuremath{#1}}}}}}

\DeclareMathOperator{\BA}{\mathbf{BA}}

\newcommand{\han}[1]{\textcolor{blue}{#1}}

\allowdisplaybreaks

\newcommand{\f}{\frac}

\begin{document}
	
	\title[The phase structure of asymmetric ballistic annihilation]{The phase structure of asymmetric ballistic annihilation}

	\author[M.~Junge]{Matthew Junge}
	\address{Matthew Junge, Department of Mathematics, Baruch College}
	\email{\texttt{Matthew.Junge@baruch.cuny.edu}}

	\author[H.~Lyu]{Hanbaek Lyu}
	\address{Hanbaek Lyu, Department of Mathematics, University of Wisconsin - Madison}
	\email{\texttt{hlyu@math.wisc.edu}}

	

	\begin{abstract}
		Ballistic annihilation is an interacting system in which particles placed throughout the real line
		move at preassigned velocities and annihilate upon colliding. The longstanding conjecture that in the symmetric three-velocity setting there exists a phase transition for the survival of middle-velocity particles was recently resolved by Haslegrave, Sidoravicius, and Tournier. We develop a framework based on a mass transport principle to analyze three-velocity ballistic annihilation with asymmetric velocities assigned according to an asymmetric probability measure. We show the existence of a phase transition in all cases by deriving universal bounds. In particular, all middle-speed particles perish almost surely if their initial density is less than $1/5$, regardless of the velocities, relative densities, and spacing of initial particles. We additionally prove the continuity of several fundamental statistics as the probability measure is varied. 
	\end{abstract}

	
	\maketitle

	\section{Introduction}
	\label{Introduction}

	\emph{Ballistic annihilation} (BA) starts with particles placed throughout the real line with independent spacings according to some probability measure $\mu$ supported on $(0,\infty)$. Each particle is independently assigned a real-valued velocity via a probability measure $\nu$. Particles simultaneously begin moving at their assigned velocity and mutually annihilate upon colliding. This relatively simple to define system has formidable long-range dependence that is both interesting and challenging to understand rigorously. 
	%
	

	The process was introduced by Elskens and Frisch with just two velocities \cite{2speeds}. Their motivation was to understand how the laws of motion affect the space-time evolution of particle types. They focused on ballistic motion with annihilation because it is an extreme case of diffusion-limited reactions that were being studied around the same time \cite{recombination, dla_fractal,anti-particle}.  A fundamental question is how $\mu$  and $\nu$ relate to the probability the particle at the origin survives for all time.
	Ballistic annihilation turned out to be a rather nuanced process and received considerable attention from physicists (see \cite{b3,b4,b5,b7,b12} for a start). 
	
	Ben-Naim, Redner, and Leyvraz in  \cite{continuous} and later Krapivsky and Sire \cite{b9} considered atomless $\nu$. They conjectured that survival probabilities respond continuously to perturbations in the velocity measure. 
	Unlike the continuous case, ballistic annihilation with $\nu$ supported on a finite set was predicted to exhibit abrupt phase transitions;
	a given particle type persists above a certain critical initial density and perishes below it \cite{discrete}. 


	With two velocities, a comparison to simple random walk ensures that the phase transition occurs when particle types are in balance. However, the three-velocity case becomes remarkably more complicated. Let $v>0$ and consider velocities in $\{-v,0,1\}$, sampled independently for each particle according to the {general probability measure}
	\begin{align}\nu(p,\lambda,v) := (1-\lambda)(1-p) \delta_{-v} + p \delta_0 + \lambda (1-p) \delta_{1}\label{eq:nu}\end{align}
	with parameters $p, \lambda \in (0,1)$. These three parameter systems are representative of all three velocity ballistic annihilation systems. Indeed, as the collision pairings in BA are invariant under translation of the underlying speed set, the system with speeds $v_1<v_2<v_3$ is equivalent to the system generated from \eqref{eq:nu} with $v=(v_1 - v_2)/(v_3 - v_2)$.

	
	We will refer to particles with velocity $0$, $+1$, and $-v$ as \textit{blockades}, \textit{right particles}, and \textit{left particles}, respectively. We call ballistic annihilation with the measure $\nu(p,1/2,1)$ the \emph{symmetric case}. Suppressing the dependence on $\mu$, let 
	$$\theta:=\theta(p,\lambda,v) = \P(\text{the origin is never visited by a moving particle}) = (1- \vec q) (1- \cev q)$$
	with
	\begin{align}
		\cev{q} &=  \P(\text{the origin is ever visited by a left particle in the half-process on $[0,\infty)$}),\\
		\vec{q} &= \P(\text{the origin is ever visited by a right particle in the half-process on $(-\infty,0]$}).
		\label{eq:def_q}
	\end{align}  
	\noindent Ergodicity ensures that a positive density of blockades are never annihilated whenever $\theta>0$. Accordingly, we say the system \textit{fixates} if $\theta(p,\lambda,v )>0$ and \textit{fluctuates} otherwise. It is natural to ask what values of $p$ result in fixation. Define
	\begin{align}
		p_c^-(\lambda,v) = \inf \{ p \colon \theta(p,\lambda,v ) > 0 \} \text { and } p_c^+(\lambda,v ) = \sup \{ p \colon \theta(p,\lambda,v ) =0 .\} \label{eq:pc+}
	\end{align}

	\begin{figure}[h]
		\centering
		\hspace{0cm}
		\includegraphics[width = 1 \linewidth]{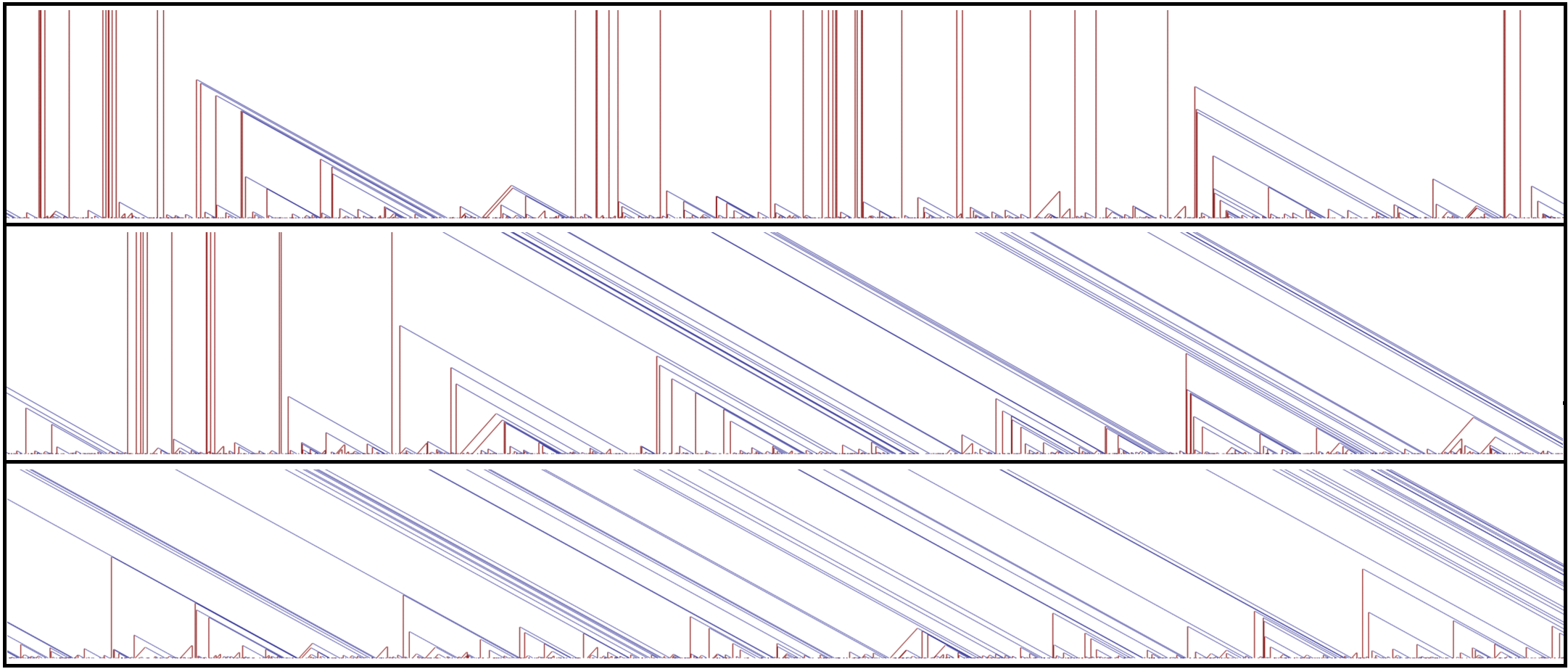}
		\vspace{-0.2cm}
		\caption{BA with spacing distribution $\mu=\text{Exp}(1)$ and 2000 particles for parameters $(\lambda,v)=(1/4,2)$ and $p=4/9$ (top), $p=3/7$ (middle), and $p=1/3$ (bottom). }
		\label{fig:simulation}
	\end{figure}

	A challenging aspect in the rigorous study of BA is that the manner in which collisions occur may become dramatically different with even a slight perturbation of the parameters.  For example, there is no known coupling that proves $\theta(v, p, \lambda)$ is monotonic in $p$. Thus, the {upper and lower} critical values in \eqref{eq:pc+} may not coincide. 
	
	In the early 1990s, physicists were interested in proving  that $p_c^-(1/2,1) >0$ for atomless $\mu$ and, even better, confirming a simple heuristic that $p_c^-(1/2,1) = 1/4 = p_c^+(1/2,1)$ \cite{b2}. This was addressed to the satisfaction of the physics community by Droz et al.\ in \cite{b4} and Krapivsky and Redner in \cite{b8}. Recently, mathematicians rediscovered this question and sought to replace derivations from \cite{b4, b8} with a rigorous probabilistic argument. Some progress towards upper bounds was made in \cite{bullets,ST,BGJ}, but a positive lower bound on $p_c^-(1/2,1)$ and a sharp upper bound remained elusive. Several major breakthroughs came in a series of papers (ultimately combined into one) from Haslegrave, Sidoravicius, and Tournier \cite{HST}. They proved that $p_c(1/2,1) = 1/4$ for all atomless $\mu$. Additionally, they confirmed the intriguing formulas from \cite{b8} for the density decay of different velocity particles as $p$ is varied.

	The situation for symmetric and atomic $\mu$ is rather different than what occurs in the atomless case. With unit spacings ($\mu=\delta_{1}$), it is possible that three particles collide simultaneously, in which case all three particles are removed from the system. Burdinski, Gupta, and Junge conjectured in \cite{BGJ} that the presence of triple collisions lowers the critical threshold. This was confirmed rigorously in \cite[Corollary 20]{HST} by proving $.2354 <p_c^- \leq p_c^+ < .2406$. The existence and location of the transition remain unknown. Note that Haslegrave, Sidoravicius, and Tournier showed that $p_c = 1/4$ in the symmetric setting where triple collisions result in the destruction of the blockade and either of the left or right particles chosen independently and uniformly at random \cite{HST}.

	The arguments in \cite{HST} for the symmetric case rely on the feature that reversing a configuration on a finite interval preserves the probability of collision events. Such symmetry under configuration reversal is lost in the general case. Our main contribution is showing that translation invariance of the process is all that is needed to perform a similar analysis. We accomplish this via a mass transport principle (Proposition~\ref{prop:MTP}). {The idea of using the mass transport principle to analyze BA} has found more applications. Haslegrave and Tournier used a mass transport principle to obtain a recursive formula for the Laplace transform of the location of the first particle to reach the origin \cite[Corollary 3]{HT}. 
	More recently, Benitez, Junge, {Lyu}, Redman, and Reeves relied heavily on the mass transport principle to study the phase transition in a coalescing version of ballistic annihilation in which collisions sometimes generate new particles \cite{CBA}. The mass transport perspective could be useful for analyzing more general particle systems. It has recently been applied to diffusion-limited annihilating systems, parking processes, and {cyclic particle systems} \cite{CRS, parking, damron2020stretched, foxall2018clustering}.

	Not much was known rigorously about the asymmetric case. Droz et al.\ \cite[Sec. IV]{b4}  provided a formula for the particle density for arbitrary velocities and weights. Getting information from the formula has a lot of implied numerical work that takes place over several steps. This makes it unclear if their solution could give a meaningful conjecture for anything but the symmetric case. Krapivsky and Redner \cite[Eq. (10)]{b8} discussed the case of three asymmetric velocities in the mean-field setting. However, they were unable to solve the corresponding system of differential equations, nor were they able to infer the location of the phase transition.
	Broutin and Marckert proved some invariance properties for a one-sided version of ballistic annihilation with finitely many particles called the \emph{bullet process} in \cite{bullets2}. Haslegrave and Tournier recently proved an independence result for asymmetric ballistic annihilation restricted to the first $n$ particles in $(0,\infty$) \cite{HT}. Namely, when $\mu$ has a gamma distribution, the distance of the $n$th particle from the origin is independent of the velocities and collision pairings for the $n$ particles in the finite system.


	\subsection{Statements of results}
	
	For each $x\in \mathbb{R}$, define random a variable $\vec{\tau}$ to be the first nonzero time at which $0$ is occupied by a right particle in the process restricted to the half-line $(- \infty , 0]$. Define $\cev{\tau}$ similarly for the process restricted to $[0, \infty)$. Define the following probabilities:
	\begin{align}\label{eq:def_alpha} 
		\vec{\alpha} &= \P(\vec{\tau} \leq  \cev{\tau}\mid \vec{\tau},\cev{\tau}<\infty ),\quad \cev{\alpha} = \P(\vec{\tau} \geq \cev{\tau}\mid \vec{\tau},\cev{\tau}<\infty ), \quad \hat{\alpha} = \P(\vec{\tau} = \cev{\tau}\mid \vec{\tau},\cev{\tau}<\infty ).
	\end{align}
	In words, $\vec \alpha$ is the probability that the origin is first visited from the left conditioned that it is eventually visited in each independent one-sided processes. {We make} a few remarks regarding these quantities. First, we have the identity $\vec{\alpha}+\cev{\alpha} = 1 + \hat \alpha.$
	Also $\vec \tau$ and the above probabilities depend on  $\mu,p,\lambda$, and $v$, but we suppress this dependence from the notation. Lastly, note that $\hat{\alpha}=0$ when $\mu$ is atomless. 
	
	%

	{Our first result is a nontrivial functional inequality that implies fluctuation.}
	
	\begin{theorem} \label{thm:main1} 
		If
		\begin{align}
			p\leq F(p,\lambda,v):= \max\left\{ \frac{\lambda(2+\hat{\alpha})-\vec{\alpha}}{\lambda(2+\hat{\alpha}) + 1+\cev{\alpha}-\vec{\alpha}}, \frac{(1-\lambda)(2+\hat{\alpha})-\cev{\alpha}}{(1-\lambda)(2+\hat{\alpha}) + 1+\cev{\alpha}-\vec{\alpha}} \right\},
			\label{eq:F}
		\end{align}
		then $\theta(p,\lambda,v)=0$.
	\end{theorem}
	
	We conjecture that fixation occurs whenever $p > F(p,\lambda,v)$. So, the phase transition is at what we believe to be the unique fixed point $p_c = F(p_c, \lambda, v)$. Proving this will likely require closed forms for $\vec \alpha$ and $\cev \alpha$. See Remark~\ref{rem:aa} for more discussion. The only setting for which a closed form is known is the symmetric setting with atomless $\mu$; it is immediately apparent that $\vec \alpha \equiv 1/2 \equiv \cev \alpha$ and so  $F(p,1/2,1) \equiv 1/4$. A regularity argument in \cite[Section 3.2]{HST} proves that the set $\{ p > 1/4\colon \vec \theta(p,1/2,1) > 0\}$ is both open and closed in $(1/4,1]$, and thus is equal to the whole interval. 
	
	We next turn to obtaining {universal} bounds on $p_c^-$ and $p_c^+$. First we describe an elementary bound. When there are almost surely no triple collisions and the probability of a given velocity being assigned exceeds $1/2$, a coupling to random walk guarantees that there will be infinitely many surviving particles with that velocity. Birkhoff's ergodic theorem (confer \cite[Proposition 3.2]{bullets} or \cite[Proposition 3.1]{ST}) ensures that there can be at most one particle type that survives with positive probability. Thus, $p>1/2$ implies $\theta(p,\lambda,v) >0$, and $\max\{\lambda(1-p),(1-\lambda)(1-p)\}> 1/2$ implies that $\theta(p,\lambda,v) =0$. This gives the bounds
	\begin{align}
		\max \left \{ \f{  2 \lambda-1}{2 \lambda}, \f{ 1- 2 \lambda}{2- 2 \lambda}\right\} \leq  p_c^-(\lambda,v) \leq  p_c^+(\lambda,v) \leq 1/2, \label{eq:elementary}
	\end{align}
	which are depicted on the left in Figure \ref{fig:phases}.

	\begin{figure}
		\begin{center}

		\begin{tikzpicture}
\begin{axis}[
            xmin=0, xmax=1, 
            ymin=0, ymax=1,
            xtick={0,1},
            ytick={1},
            xlabel=$\lambda$, ylabel=$p$,
            xlabel style={above},
             ylabel style={below},
            scale=.75
            ]
 
\addplot [name path = A1,
    domain = .0001:.4999,
    samples = 1000] {(1-2*x)/(2*(1-x)};
    
\addplot [name path = A2,
    domain = .5001:.9999,
    samples = 1000] {(2*x-1)/(2*x)};
 
\addplot [name path = BM,
    domain = 0:1] {1/2};
    
\addplot [name path = BH,
    domain = 0:1] {1};

\addplot [name path = BL,
    domain = 0:1] {0};
    
\addplot [gray!100] fill between [of = BM and BH, soft clip={domain=0:1}];

\addplot [gray!50] fill between [of = A1 and BL, soft clip={domain=.0001:.4999}];

\addplot [gray!50] fill between [of = A2 and BL, soft clip={domain=.5001:.9999}];


\draw[gray!100, very thick] (1/2,1/2) -- (1/2,1/4);
\draw[gray!50, very thick] (1/2,1/4) -- (1/2,0);
 
\end{axis}
\end{tikzpicture}
\qquad 
		\begin{tikzpicture}
\begin{axis}[
            xmin=0, xmax=1, 
            ymin=0, ymax=1,
            xtick={0,1},
            ytick={1},
            scale=.75
            ]

\node[left] at (-.2,1/6) {$1/5$};
\node[left] at (-.2,1/4) {$1/4$};

\addplot [name path = A1,
    domain = .0001:3/8,
    samples = 1000] {(1-2*x)/(2*(1-x)};
    
\addplot [name path = A3,
    domain = 3/8:5/8,
    samples = 1000] {1/5)};
    
\addplot [name path = A2,
    domain = 5/8:.9999,
    samples = 1000] {(2*x-1)/(2*x)};
 
\addplot [name path = BM,
    domain = 0:1] {max(x/(1+x), (1-x)/(2-x))};

\addplot [name path = BH,
    domain = 0:1] {1};

\addplot [name path = BL,
    domain = 0:1] {0};
    
\addplot[dashed,domain = 1/3:2/3] {1/4};
    
\addplot [gray!100] fill between [of = BM and BH, soft clip={domain=0:1}];

\addplot [gray!50] fill between [of = A1 and BL, soft clip={domain=.0001:3/8}];

\addplot [gray!50] fill between [of = A3 and BL, soft clip={domain=3/8:5/8}];

\addplot [gray!50] fill between [of = A2 and BL, soft clip={domain=5/8:.9999}];

\draw[gray!50, very thick] (3/8,1/5) -- (3/8,0);
\draw[gray!50, very thick] (5/8,1/5) -- (5/8,0);

\draw[gray!100,very thick] (1/2,1/3) -- (1/2,1/4);
\draw[gray!50, very thick] (1/2,1/4) -- (1/2,0);

\end{axis}
\end{tikzpicture}

			\caption{The left figure has the elementary bounds for fixation $\theta >0$ (dark gray) and fluctuation $\theta =0$ (light gray) regimes from \eqref{eq:elementary}. The phase behavior in white regions is unknown. The central line depicts the phase transition in the symmetric case at $p_c(1/2,1/2) =1/4$  from \cite{HST}. The right figure has the improved bounds from Theorem \ref{thm:main2}. When $\mu$ is atomless, the lower bound on fluctuation improves from the solid line at $p=1/5$ to the dashed line at $p=1/4$. 
			}\label{fig:phases}
		\end{center}
	\end{figure}
	
	{For the function $F$ defined in \eqref{eq:F}, we show that $F(p,\lambda, v) \geq 1 /(4+\hat \alpha)$ (see the proof of Theorem \ref{thm:main2}). Hence by Theorem \ref{thm:main1}, it follows that $p \le 1 /(4+\hat \alpha)$ implies fluctuation. This significantly improves the lower bound in \eqref{eq:elementary}, matching the critical value $p_{c}(1/2,1)=1/2$ in the totally symmetric case with atomless $\mu$.} Using ideas from \cite{HST} we also derive a universal upper bound. This tells us that any transitions between fluctuation and fixation must occur in the middle region of the phase diagram depicted in Figure \ref{fig:phases}. 
	{The precise universal bounds are given in Theorem \ref{thm:main2}.}
	
	\begin{theorem} 
		\label{thm:main2} Define the functions 
		\begin{align}\label{eq:def_pc_bounds}
			f_{*}^{(1)}(\lambda) &=  \max \left( \frac{1-2\lambda}{2-2\lambda},\, \frac{1}{5}, \,\frac{2\lambda-1}{2\lambda}   \right),\; f_{*}^{(2)}(\lambda) = \max\left( \frac{1-2\lambda}{2-2\lambda},\, \frac{1}{4}, \,\frac{2\lambda-1}{2\lambda}   \right), \,  f^{*}(\lambda) = \max\left( \f{1-\lambda}{2-\lambda}, \f{\lambda}{1+\lambda} \right).
		\end{align} 
		For all $v>0$ it holds that
		\begin{align}
			f_{*}^{(1)}(\lambda)\le p_{c}^-(\lambda,v)\le p_c^+(\lambda,v) \leq  f^{*}(\lambda). 
		\end{align} 
		Furthermore, if $\mu$ is atomless, then $f_{*}^{(2)}(\lambda,v)\le p_{c}^-(\lambda,v)$. In particular, $\theta(p,\lambda,v)=0$ whenever $p\le 1/5$ for all $\lambda$, $v$, and $\mu$; this holds whenever $p\le 1/4$ if $\mu$ is atomless. 
	\end{theorem}

	As mentioned earlier, it is generally difficult to prove regularity properties such as monotonicity and continuity for probabilities associated with BA. 	In the symmetric case,  $\cev{\beta}=\vec{\beta}=0$ and $\cev{\alpha}=\vec{\alpha}=1/2$ for all $p$, so continuity of $\theta$ in $p$ is the main interest. In \cite{HST}, the authors deduced continuity of $\theta$ as a consequence of an explicit (continuous) formula in $p$. However,  the number of unknown quantities in the asymmetric case exceeds the number of equations that relate these quantities. {Indeed, our main recursions for the general case (see \eqref{eq:first}--\eqref{eq:third}) relate six quantities, whereas the recursion for the totally symmetric case in \cite{HST} (see \eqref{eq:HST}) involves only one. }

	Our final result establishes the continuity of various quantities including the survival probabilities of each particle type. Define the following survival probabilities of moving particles:
	\begin{align}
		\vec{\beta}&=\P(\textup{a particle is right-moving and is never annihilated}), \\
		\cev{\beta}&=\P(\textup{a particle is left-moving and is never annihilated}).
	\end{align}
	
	\begin{theorem} \label{thm:continuity}
		For fixed $\mu$ and $v>0$, it holds that $\theta$, $\cev{\beta}$,  $\vec{\beta}$, and the product $\cev{q}\vec{q}\hat{\alpha}$ are continuous in $p$ and $\lambda$. {Furthermore, in the fixation regime $\{ (p,\lambda)\,|\, \theta>0 \}$, all quantities $\cev{q},\vec{q},\vec{q}, \cev{\alpha}, \vec{\alpha}$, and $\hat{\alpha}$ are continuous in $p$ and $\lambda$.}
	\end{theorem}

	It may not be true in general that $\theta$ is continuous in $v$. Consider the situation that $\mu = \delta_1$ so that particles are initially placed on the integers. If $v$ is irrational, then there are no triple collisions. It follows that $\hat \alpha = 0$ for irrational $v$.  However, if $v$ is rational, then there is a positive probability that two particles arrive simultaneously to the origin. Thus, $\hat \alpha$ is discontinuous at all values of $v$. While this does not prove that $\theta$ is discontinuous, it suggests that the function $F$ from \eqref{eq:F} is also discontinuous in $v$. Since we believe that $p_c$, assuming it exists, solves $p_c = F(p_c,\lambda,v)$, it seems possible that both $p_c$ and $\theta$ are discontinuous in $v$ for this case.

	\subsection{Notation} \label{sec:notation}

	{An \textit{initial configuration} is a placement of particles on $\mathbb{R}$ together with their velocities.} For any collision event $E$ and an interval $I\subseteq \mathbb{R}$, denote by $E_{I}$ the event that $E$ occurs in the process after restricting to only the particles in the initial configuration on $I$. 
	
    Using translation invariance of the process on $\mathbb{R}$, we may assume that the origin is always occupied by a particle. For each $i\in \mathbb{N}$, we denote the $i$th particle on the positive (resp., negative) real line by $\bullet_{i}$ (resp., $\bullet_{-i}$). Denote the event that the $i$th particle is a blockade, right particle, and left particle by $\{ \B_{i} \}$, $\{ \R_{i}  \}$, and $\{ \L_{i} \}$, respectively. We write $x_{i}\in \mathbb{R}$ for the initial location of $\bullet_i$. For any $i,j\in \mathbb{Z}$ and $x\in \mathbb{R}$, denote 
	\begin{align}
		\{ \R_{i}\rightarrow \B_{j} \} &= \{ \R_{i} \}\cap \{ \B_{j} \} \cap \{ \text{the $i$th particle annihilates with the $j$th particle} \}\\
		\{ \R_{i}\rightarrow \B \} &= \{ \R_{i} \}\cap \{ \text{the $i$th particle annihilates with a blockade} \} \\
	    \{ \R_{i}\leftrightarrow \L_j \} &= \{ \R_{i} \}\cap \{\L_j\} \cap \{ \text{the $i$th particle annihilates with the $j$th particle} \} \\
        \{ \R_{i}\leftrightarrow \L \} &= \{ \R_{i} \} \cap \{ \text{the $i$th particle annihilates with a left particle} \} \\
		\{ \R_i \rightarrow \B \leftarrow \L_j\} & = \{\text{there is a triple collision between $\R_i$, a blockade, and $\L_j$}\}\\
		\{ x \leftarrow \L \}  &= \{ \text{site $x$ is visited by a left particle after time 0} \}\\
		\{ x \leftarrow_{1} \L_{n} \}  &= \{ \text{the $n$th left particle is the first left particle that visits site $x$ after time 0} \}
	\end{align}
	and we use the similar notation for other types of collisions and visits. The quantities $\vec{q}$, $\cev{q}$, $\vec \alpha, \cev \alpha,$ and  $\hat \alpha$ are defined in \eqref{eq:def_q} and \eqref{eq:def_alpha}. Define the survival probabilities for moving particles by $\vec{\beta} = \P(\R_{1}\nrightarrow \bullet)$ and $\cev{\beta} = \P(\bullet \nleftarrow \L_{1})$, which were also defined above Theorem \ref{thm:continuity}. 
	
	For each interval $I\subseteq \mathbb{R}$, denote the process restricted on $I$ by $\textbf{BA}[I]$. For each $a,b\in \mathbb{Z}$ with $a\le b$, we define $\vec{N}(a,b)$, $W(a,b)$, and $\dot{N}(a,b)$ to be the number of surviving right particles, left particles, and blockades in $\textbf{BA}[x_{a},x_{b}]$. We also define the random variables
	\begin{align}\label{eq:W_def_1}
	 & W(a,b) = \dot{N}(a,b) -  \cev{N}(a,b) - \vec{N}(a,b), \\
	 &\cev W(a,b) = \dot N(a,b) - \cev N(a,b), \qquad  \vec W(a,b) = \dot N(a,b) - \vec N(a,b)
	\end{align}
	that track the net difference between blockades and moving particles surviving from $\textbf{BA}[x_{a},x_{b}]$.


	
	\subsection{Overview of proofs}
	The key to proving Theorem \ref{thm:main1} is Lemma \ref{lemma:main_recursion}, which gives algebraic equations relating the main probabilities associated to BA. The proof uses a mass transport principle to generalize the approach in \cite[Proposition 5]{HST}. The equations are further manipulated in Proposition \ref{prop:trichotomy} to obtain recursive formulas for $\vec q$ and $\cev q$. This leads to the criterion for fluctuation in Theorem \ref{thm:main1}. It is noteworthy that our methods, which average over the longtime behavior of all particles in the system, are not sensitive to the value of $v$. Any dependence between $v$ and the phase behavior is contained implicitly in the quantities $\vec \alpha, \cev \alpha, \hat \alpha$, $\vec \beta$, and $\cev \beta$.

	The proof of Theorem \ref{thm:main2} uses Theorem \ref{thm:main1} as well as Lemma \ref{lem:vecW}, which is an adaptation of \cite[Proposition 11]{HST}. The lemma provides a necessary and sufficient condition for fixation to occur in terms of $\cev W$ and $\vec{W}$ from \eqref{eq:W_def_1}:
	\begin{align}\label{eq:HSTtheta}
        \theta>0 \quad \text{ if and only if } \quad \exists k,\ell \,\, s.t. \,\, \text{$\E[ \cev{W}(1,k) ]>0$ and  $\E[\vec{W}(-\ell, -1)]>0$}.
    \end{align}
	The argument in \cite{HST} relies on symmetry to deduce that left and right particles almost surely annihilate for all $p$. The asymmetric setting lacks this feature. To overcome this, we show that for all parameter choices, either [$\vec{\beta}=0$ and $\vec{q}\le \cev{q}$] or [$\cev{\beta}=0$ and $\vec{q}\ge \cev{q}$] hold (see Proposition \ref{lemma:alpha_half_implications}). This dichotomy allows us to apply similar reasoning as in \cite{HST}.

	To deduce continuity of $\theta$ in Theorem~\ref{thm:continuity}, we sharpen \eqref{eq:HSTtheta} to a quantitative characterization of $\theta$ in terms of the expected value of $W(1,k)$ from \eqref{eq:W_def_1}. Specifically, Lemma \ref{lemma:theta} states that
	\begin{align} \theta = 0 \vee  \textstyle \sup_{k\ge 1}\, k^{-1} \E[W(1,k)] . \label{eq:theta_e}
	\end{align} 
	Notice that instead of $\cev W$ and $\vec{W}$, we work with $W$, which subtracts the number of left \emph{and} right surviving particles from the total number of blockades that survive in ballistic annihilation restricted to the particles in $[x_1, x_k]$.  
	
	 The advantage of using $W$ is that a more general superadditive property holds when combining its values on adjacent intervals (see Proposition~\ref{prop:superadditivity_Z}). This was first observed in \cite{CBA} for coalescing ballistic annihilation. The benefit of working directly with $W$ is that we can compare the event of blockade survival to a random walk whose steps are distributed as $W(1,k)$ for a deterministic $k$ rather than some random $K$ as employed in \cite{HST}. This allows us  to obtain the quantitative formula for $\theta$ at \eqref{eq:theta_e}, which is essential to our argument for proving its continuity. We elaborate a bit more on this difference in Remark \ref{rem:theta}. 
	 
	Finally, to prove Theorem \ref{thm:continuity}, we first derive continuity of $\theta$ from Lemma \ref{lemma:theta} and then show that $\cev{\beta}-\vec{\beta}$ is continuous by a mass transport principle (see Proposition \ref{prop:MTP_continuity}). We then use the classical pasting lemma (Lemma \ref{lemma:pasting}) to show $\cev{\beta}$ and $\vec{\beta}$ are continuous. The second part of the statement is deduced from the first part by using $\theta=(1-\cev{q})(1-\vec{q})$ and lower semi-continuity of the quantities $\cev{q},\vec{q}, \cev{q}\vec{q}\cev{\alpha},$ and  $\cev{q}\vec{q}\vec{\alpha}$ (see Proposition \ref{prop:LSC}).
	


	\subsection{Organization} 
	In Section \ref{section:recursion} we prove Lemma \ref{lemma:main_recursion} using the mass transport principle and use it to deduce Theorem \ref{thm:main1}. In Section \ref{sec:theta_lemma}, we prove lemmas that relate blockade survival to various finite conditions, and also derive Theorem \ref{thm:main2}. Lastly,  Section \ref{sec:continuity} gives the proof of Theorem \ref{thm:continuity}.
	
	\section{The mass transport principle and proof of Theorem \ref{thm:main1}} \label{section:recursion}
	
	In this section we prove the following lemma and then use it to deduce Theorem \ref{thm:main1}.
	\begin{lemma}\label{lemma:main_recursion}
		For arbitrary spacing distributions and parameter choices, the following equations hold:
		\begin{align}
			\hspace{0.59cm}(\cev{q}-1)( p\cev \alpha  \vec{q} \cev{q} +p\cev{q} - (1-\lambda)(1-p) ) &= \cev{q}\vec{\beta}  \label{eq:first}\\
			\hspace{1.38cm}(\vec{q}-1)( p \vec \alpha \vec{q} \cev{q} + p\vec{q} - \lambda(1-p) ) &= \vec{q}\cev{\beta} \label{eq:second} \\
			p\vec{q}\cev{q}(\cev{\alpha}-\vec{\alpha}) + p(\cev{q}-\vec{q}) -(1-2\lambda)(1-p)  &=  \vec{\beta} - \cev{\beta}. \label{eq:third}
		\end{align}
	\end{lemma}
	
	We have seven (six if $\mu$ is atomless) unknown quantities that are related by the recursions in Lemma \ref{lemma:main_recursion}. It is not possible with Lemma \ref{lemma:main_recursion} alone to obtain closed form solutions. In contrast, for the symmetric case with an atomless spacing distribution the equations in Lemma \ref{lemma:main_recursion} reduce to \begin{align}\label{eq:HST}
		(q-1)( q^{2} +2q +1- p^{-1}) = 0,
	\end{align}
	where $q=\cev{q}=\vec{q}$. This yields the following dichotomy 
	\begin{align}
		q=1 \quad \text{or} \quad q=p^{-1/2}-1.  \label{eq:dichotomy}
	\end{align}
	Hence $p\le 1/4$ implies $q=1$. 
	An analogue of \eqref{eq:dichotomy} for the general case is given in Proposition \ref{prop:trichotomy}. 
	Nonetheless, Lemma \ref{lemma:main_recursion} plays a crucial role in establishing universal bounds in the phase diagram for the asymmetric BA.

	
	The key tool in proving Lemma \ref{lemma:main_recursion} is the following mass transport principle. 
	
	\begin{prop}[Mass transport principle]\label{prop:MTP}
		Define a non-negative random variable $Z(a,b)$ for integers $a,b\in \mathbb{Z}$ such that its distribution is diagonally invariant under translation, i.e., for any integer $d$,  $Z(a+d,b+d)$ has the same distribution as $Z(a,b)$. Then for each $a\in \mathbb{Z}$,
		\begin{equation}
			\E \left[ \sum_{b\in \mathbb{Z}} Z(a,b) \right] = \E \left[ \sum_{b\in \mathbb{Z}} Z(b,a)\right].
		\end{equation}
	\end{prop}
	
	\begin{proof}
		Using  linearity of expectation and translation invariance of $\E[Z(a,b)]$, we get 
		\begin{equation}
			\E \left[ \sum_{b\in \mathbb{Z}} Z(a,b) \right] =  \sum_{b\in \mathbb{Z}} \E [Z(a,b)]  = \sum_{b\in \mathbb{Z}} \E [Z(2a-b,a)] = \sum_{b\in \mathbb{Z}} \E [Z(b,a)] = \E  \left[ \sum_{b\in \mathbb{Z}} Z(b,a) \right].
		\end{equation}  
	\end{proof}

	\begin{prop}\label{prop:p_pm}
		The following hold:
		\begin{align}
			\P(\R_{1}\rightarrow \bullet) &= p\vec{\alpha}\vec{q}\cev{q}+p\vec{q}(1-\cev{q}) + \P(\R_{1}\leftrightarrow \L) \label{eq:1}\\
			\P(\bullet \leftarrow \L_{-1}) &= p\cev{\alpha}\vec{q}\cev{q} + p(1-\vec{q})\cev{q}+\P(\R\leftrightarrow \L_{-1})\label{eq:2} \\
			\P(\R_{1}\leftrightarrow \L) & = \P(\R\leftrightarrow \L_{-1}). \label{eq:3} 
		\end{align} 
	\end{prop}

	\begin{proof}
		For $a,b\in \mathbb{Z}$, define the indicator variable
		$Z(a,b) = \1\left( \{\R_{a}\rightarrow \bullet_{b}\} \right).$
		Note that 
		\begin{align}
			\E  \left[ \sum_{b\in \mathbb{Z}} Z(1,b) \right] = \P(\R_{1}\rightarrow \bullet).
		\end{align}
		Using translation invariance of BA,
		\begin{align}
			\E \left[ \sum_{b\in \mathbb{Z}} Z(b,1)\right] &= \P(\R \rightarrow \B_{1})  + \P(\R\leftrightarrow \L_{1} )
			= p\vec{q}\cev{q}\vec{\alpha} +p\vec{q}(1-\cev{q}) + \P(\R_{1}\leftrightarrow \L). 
		\end{align}
		The equality $\P(\R \rightarrow \B_{1}) = p\vec{q}\cev{q}\vec{\alpha} +p\vec{q}(1-\cev{q})$ follows by conditioning on whether $x_1$ is visited from both sides or not in the processes on $(-\infty,x_1]$ and $[x_1,\infty)$.
		Proposition \ref{prop:MTP} then yields \eqref{eq:1}. A similar argument shows \eqref{eq:2}. 
		
		For \eqref{eq:3}, we use the mass transport principle for the indicators 
		$$\tilde{Z}(a,b) = \1\left( \{\R_{a}\rightarrow \L_{b}\} \right).$$
		By translation invariance we have 
		\begin{align}
			\P(\R_{1}\leftrightarrow \L) = \E \left[ \sum_{b\in \mathbb{Z}} \tilde{Z}(1,b) \right]   = \E \left[ \sum_{a\in \mathbb{Z}} \tilde{Z}(a,1) \right] = \E \left[ \sum_{a\in \mathbb{Z}} \tilde{Z}(a,-1) \right]= \P(\R\leftrightarrow \L_{-1}).
		\end{align}
		This shows the assertion. 
	\end{proof}

	\begin{prop}\label{prop2}
		$\P( (0\leftarrow \L)_{(0,\infty)} \land (\R_{1} \rightarrow \B)_{(0,\infty)}) = p\vec{q}\cev{q}( \vec{\alpha} - \hat \alpha)  + p \vec q \cev q^2 \hat \alpha $.
	\end{prop}
	
	\begin{proof}
		For integers $a,b\in \mathbb{Z}$, define the indicator variable $Z(a,b)$ by
		\begin{align}\label{eq:MTP2}
			Z(a,b) &= \1\left( (\R_{a}\rightarrow \B_{b}\not \leftarrow \L)_{[x_{a},\infty)} \land   (x_{b}\leftarrow \L)_{(x_{b},\infty)}   \right) \\
			& \qquad + \1\left( (\R_{a}\rightarrow \B_{b} \leftarrow \L) _{[x_{a},\infty)} \land   (x_{b}\leftarrow \L)_{(x_{b},\infty)}   \right).
		\end{align} 
		{It is easy to see that these indicator variables are diagonally invariant under translation.} Observe that 
		\begin{align}
			\E \left[ \sum_{b\in \mathbb{Z}} Z(1,b)\right] &=\E\bigg[ \1\left(\exists b\in \mathbb{Z}\colon (\R_{1}\rightarrow \B_{b} \not \leftarrow \L)_{[x_{1},\infty)} \land  (x_{b}\leftarrow \L)_{[x_{b},\infty)}   \right) + \\
			& \qquad \qquad  \1 \left( \exists b \in \mathbb Z ( \R_1 \rightarrow  \B_b \leftarrow  \L)_{[x_1, \infty)} \land (x_b \leftarrow \L )_{(x_b, \infty)} \right) \bigg]\\
			& = \E \left[ \1\left( (\R_{1} \rightarrow \B)_{[x_{1},\infty)} \land (0\leftarrow \L)_{[x_{1},\infty)} \right) \right]\\
			&= \P( (0\leftarrow \L)_{(0,\infty)} \land (\R_{1} \rightarrow \B)_{(0,\infty)}). \label{eq:start}
		\end{align}
		
		The sum of the indicator random variables decomposes the two ways---double or triple collisions involving $\b_1$---that 0 can be visited on the event $(\R_1 \rightarrow \B)$. 
		On the other hand, first note that for any $b\in \mathbb{Z}_{<0}$, 
		\begin{align}
			\{\R_{b}\rightarrow \B_{1} \not \leftarrow \L  \}_{_{[x_{b},\infty)}} &= \{\R_{b}\rightarrow \B_{1}  \}_{_{(-\infty,x_{1}]}} \cap \{ \vec{\tau}^{(x_{1})}<\cev{\tau}^{(x_{1})} \},\\
			\{\R_{b}\rightarrow \B_{1}  \leftarrow \L  \}_{_{[x_{b},\infty)}} &= \{\R_{b}\rightarrow \B_{1}  \}_{_{(-\infty,x_{1}]}} \cap \{ \vec{\tau}^{(x_{1})}=\cev{\tau}^{(x_{1})} \}.
		\end{align}
		From this it follows that 
		\begin{align}
			\E \left[ \sum_{b\in \mathbb{Z}} Z(b,1) \right]&= \E\bigg[ \1\left(\exists b\in \mathbb{Z}\colon (\R_{b}\rightarrow \B_{1} \not \leftarrow \L)_{[x_{b},\infty)} \land  (x_{1}\leftarrow \L)_{(x_{1},\infty)}   \right) \\
			& \qquad \qquad + \1 \left( \R_b \to \B_1 \leftarrow \L)_{[x_b , \infty)} \land ( x_1  \leftarrow \L )_{(x_1,\infty)}  \right)\bigg] \\
			&= \E \bigg[ \1\left((\R\rightarrow \B_{1} \not \leftarrow \L)_{(-\infty,x_{1}]} \land  (x_{1}\leftarrow \L)_{[x_{1},\infty)}  \land ( \vec{\tau}^{(x_{1})}<\cev{\tau}^{(x_{1})} ) \right)\\
			& \qquad \qquad + \1 \left( \R \to \B_1 \leftarrow \L)_{[- \infty, \infty )} \land ( x_1  \leftarrow \L )_{(x_1,\infty)} \land \vec{\tau}^{(x_{1})}=\cev{\tau}^{(x_{1})}  \right)\bigg]\\
			&= p \vec q \cev q ( \vec{\alpha} - \hat \alpha) + p \vec q \cev{q}^{2} \hat \alpha. \label{eq:end}
		\end{align}
		Proposition~\ref{prop:MTP} gives the claimed equality between \eqref{eq:start} and \eqref{eq:end}.  
	\end{proof}

	Now we are ready to prove Lemma \ref{lemma:main_recursion}.
	
	\begin{proof}[\textbf{Proof of Lemma \ref{lemma:main_recursion}}]
		To show \eqref{eq:first}, we write 
		\begin{align}
			\cev{q} & = \P\left( (0\leftarrow \L)_{[0,\infty)} \right)\\
			& = \P\left( (0\leftarrow \L)_{[0,\infty)} \land (\R_{1})\right) + p\P\left( (0\leftarrow \L)_{[0,\infty)} \mid (\B_{1})\right) \\
			&\hspace{1cm}+ (1-\lambda)(1-p)\P\left( (0\leftarrow \L)_{[0,\infty)} \mid (\L_{1})\right) \\
			& =\P\left( (0\leftarrow \L)_{[0,\infty)} \land (\R_{1})\right) + p\cev{q}^{2} + (1-\lambda)(1-p). \label{eq:end2}
		\end{align}
		We claim that 
		\begin{align}
			\P\left( (0\leftarrow \L)_{[0,\infty)} \land  (\R_{1})\right) = 
			p\vec{q}\cev{q}(\vec{\alpha} - \hat \alpha)  + p \vec q \cev q^2 \hat \alpha + \lambda(1-p)\cev{q} - p\vec{q}\cev{q}(1-\cev{q})-p\vec{q}\cev{q}^{2}\vec{\alpha}-\cev{q}\vec{\beta}. \label{eq:expanded}
		\end{align}
		The assertion follows from substituting \eqref{eq:expanded} into \eqref{eq:end2} and some routine algebra:
		\begin{align}
			\cev q &= p\vec{q}\cev{q}(\vec{\alpha} - \hat \alpha)  + p \vec q \cev q^2 \hat \alpha + \lambda(1-p)\cev{q} - p\vec{q}\cev{q}(1-\cev{q})-p\vec{q}\cev{q}^{2}\vec{\alpha}-\cev{q}\vec{\beta}  + p \cev q^2 + (1-\lambda)(1-p) \\
			&= p\vec{q}\cev{q}(\vec{\alpha} - \hat \alpha -1)  + p \vec q \cev q^2( \hat \alpha +1 - \vec \alpha ) + \lambda(1-p)\cev{q}  -\cev{q}\vec{\beta}  + p \cev q^2 + (1-\lambda)(1-p).
		\end{align}
		Rearranging, applying the identity $\vec \alpha + \cev \alpha = 1 + \hat \alpha$ twice, and factoring gives \eqref{eq:first}:
		\begin{align}
			\cev{q}\vec{\beta} &=  p(  1+ \vec q \cev \alpha) \cev q^2   + [\lambda(1-p) -p\vec{q} \cev \alpha -1] \cev{q}   + (1-\lambda)(1-p) \\
			&= (\cev q -1) ( p \cev q + p \cev \alpha \vec q \cev q - (1-\lambda)(1-p) ).
		\end{align}

		Next we show that \eqref{eq:expanded} holds. By partitioning on the possible collisions of $\R_{1}$ and using Proposition \ref{prop2},
		\begin{align}
			\P\left( (0\leftarrow \L)_{[0,\infty)} \land (\R_{1})\right) &= \P\left( (0\leftarrow \L)_{[0,\infty)} \land (\R_{1}\rightarrow \B)_{[x_{1},\infty)}\right) \\
			&\qquad +  \P\left( (0\leftarrow \L)_{[0,\infty)} \mid  (\R_{1}\rightarrow \L)_{[x_{1},\infty)}\right) \P\left( (\R_{1}\rightarrow \L)_{[x_{1},\infty)} \right) \\
			& \qquad \qquad +\P\left( (0\leftarrow \L)_{[0,\infty)} \land (\R_{1}\nrightarrow \bullet) \right)\\
			& = p\vec{q}\cev{q}\vec \alpha +  p \vec q \cev q^2 \hat \alpha  +\cev{q}\P\left( (\R_{1}\rightarrow \L)_{[x_{1},\infty)}\right) +0.
		\end{align}
		By \eqref{eq:1}, we have 
		\begin{align}
			\P\left( (\R_{1} \rightarrow \B)_{[x_{1},\infty)} \right) = p\vec{q} (1-\cev{q}) + p\vec{q}\cev{q}\vec{\alpha}  .\label{eq:RB}
		\end{align}
		Rearranging the identity 
		$$\P(\R_1) = \lambda (1-p) = \P( (\R_1 \rightarrow \B)_{[x_1, \infty)}) +\P( (\R_1 \rightarrow \L)_{[x_1, \infty)} )+ \vec \beta, $$ 
		and using \eqref{eq:RB} gives
		\begin{align}
			\P\left( (\R_{1}\rightarrow \L)_{[x_{1},\infty)}\right)&=   \lambda(1-p) - \P\left( (\R_{1}\rightarrow \B)_{[x_{1},\infty)}\right) - \vec{\beta}  \\
			&=   \lambda(1-p) - p\vec{q}(1-\cev{q})-p\vec{q}\cev{q}\vec{\alpha}  -\vec{\beta} \label{eq:p4_eqation}
		\end{align}
		Combining the two equations gives the claim. A symmetric argument gives \eqref{eq:second}.

		To show \eqref{eq:third}, we observe that \eqref{eq:1} implies
		\begin{align}
			\P(\R_{1}\rightarrow \L) = \lambda(1-p) - \vec{\beta}- \Big( p\vec{\alpha}\vec{q}\cev{q}+p\vec{q}(1-\cev{q}) \Big)  
		\end{align}
		which, by \eqref{eq:3}, is equal to
		\begin{align}
			\P(\R\leftarrow \L_{-1})= (1-\lambda)(1-p) - \cev{\beta} - \Big( p\cev{\alpha}\vec{q}\cev{q}  +p(1-\vec{q})\cev{q} \Big).
		\end{align}
		Combining these equations gives \eqref{eq:third}. 
	\end{proof}

	Next, we use Lemma \ref{lemma:main_recursion} to show that the tuples $(\cev{\alpha},\cev{q})$, $(\cev{\alpha},\cev{q})$, and $(\cev{\alpha},\vec{\alpha},\cev{q},\vec{q})$ satisfy certain polynomial equations depending on whether $\cev{q}$ or $\vec{q}$ equal 1 or all moving particles vanish almost surely.
	
	
	\begin{prop}\label{prop:trichotomy}
		For an arbitrary spacing distribution and parameters, the following hold:
		\begin{enumerate}[label = (\roman*), leftmargin=*]
			\item If $\vec{q}=1$, then $ (\vec{\alpha}, \cev{q})$ satisfy
			\begin{equation}\label{thm1:eq1}
				p\vec{\alpha} \cev{q}^{2} - (p\cev{\alpha}+\lambda(1-p)) \cev{q} + (1-\lambda)(1-p)=0.
			\end{equation}
			\item If $\cev{q}=1$, then $(\cev{\alpha}, \vec{q})$ satisfy
			\begin{equation}\label{thm1:eq2}
				p\cev{\alpha} \vec{q}^{2} - (p\vec{\alpha}+(1-\lambda)(1-p)) \vec{q} + \lambda (1-p)=0.
			\end{equation}
			\item If $\min(\vec{q},\cev{q})<1$ and $\vec{\beta}=\cev{\beta}=0$, then $(\cev{\alpha},\vec{\alpha},\cev{q},\vec{q})$ satisfy
			\begin{align}\label{thm1:eq3}
				\begin{cases}
					p\cev{\alpha}\vec{q}^{2} + \left( ((1-\lambda)\vec{\alpha}-\lambda \cev{\alpha})(1-p)+p \right) \vec{q} - \lambda(1-p)=0 \\
					p\vec{\alpha}\cev{q}^{2} + \left( (\lambda \cev{\alpha}-(1-\lambda)\vec{\alpha})(1-p)+p \right)\cev{q}- (1-\lambda)(1-p)=0
				\end{cases}.
			\end{align}
		\end{enumerate}
	\end{prop}

	\begin{proof}
		First suppose $\vec{q}=1$. Then the second equation in Lemma \ref{lemma:main_recursion} implies $\cev{\beta}=0$. Plugging in these values to \eqref{eq:first} and \eqref{eq:third} gives 
		\begin{align}\label{eq:p_pm_pf1}
			p(1+\vec{\alpha}) \cev{q}^{2} - (p\vec{\alpha} +\lambda)+1-\lambda+\vec{\beta}) \cev{q} + (1-\lambda)(1-p)=0
		\end{align}
		and 
		\begin{align}\label{eq:p_pm_pf2}
			\vec{\beta}= p\cev{q}(\cev{\alpha}-\vec{\alpha}) + p(\cev{q}-1) - (1-2\lambda)(1-p).
		\end{align}
		Combining these two equations gives \eqref{thm1:eq1}. A similar argument shows (ii). 
		
		To show (iii), first note that Lemma \ref{lemma:main_recursion} with $\vec{\beta}=\cev{\beta}=0$ implies 
		\begin{align}\label{eq:pf_lemma4_eq3}
			(\cev{q}-1)( p\cev{\alpha} \vec{q} \cev{q} +p\cev{q} - (1-\lambda)(1-p) ) &= 0 \\
			(\vec{q}-1)( p \vec{\alpha} \vec{q} \cev{q} + p\vec{q} - \lambda(1-p) ) &= 0 \\
			p\vec{q}\cev{q}(\cev{\alpha}-\vec{\alpha}) + p(\cev{q}-\vec{q}) -(1-2\lambda)(1-p)  &=  0.
		\end{align}
		Without loss of generality we assume $\vec{q}<1$. Hence the second factor of the second equation in \eqref{eq:pf_lemma4_eq3} is zero. Adding this equation to the last equation in \eqref{eq:pf_lemma4_eq3} shows that the second factor of the first equation in \eqref{eq:pf_lemma4_eq3} is zero. Hence we have  
		\begin{align}\label{eq:pf_J_closed_eq2}
			p\cev{\alpha} \vec{q} \cev{q} +p\cev{q} - (1-\lambda)(1-p)  &= 0 \\
			p \vec{\alpha} \vec{q} \cev{q} + p\vec{q} - \lambda(1-p)  &= 0.
		\end{align}
		Solving the first equation for $\cev{q}$ and plugging into the second yields the first equation in \eqref{thm1:eq3}. A similar argument for $\cev{q}$ shows the second equation in \eqref{thm1:eq3}. 
	\end{proof}
	{
	Note that from Proposition \ref{prop:trichotomy}, it would be possible to obtain explicit formulas for $\cev{q}$ and $\vec{q}$ if we had explicit formulas for $\cev{\alpha}$ and $\vec{\alpha}$. This is the case for the totally symmetric case, but not in the general asymmetric case. However, it is still enough for us to derive sufficient conditions for fluctuation, as stated in  Theorem \ref{thm:main1}, which we justify below. }
	
	\begin{proof}[{Proof of Theorem \ref{thm:main1}}]
		Suppose $\vec{q},\cev{q}<1$.  Then, by Birkhoff's ergodic theorem, $\vec{\beta}=\cev{\beta}=0$. By Proposition \ref{prop:trichotomy} (iii), $\vec{q}$ and $\cev{q}$ are the unique positive solutions of \eqref{thm1:eq3}. Denote by $\vec{f}(t)$ the quadratic polynomial in $\vec{q}$ in the left hand side of the first equation of \eqref{thm1:eq3}, that is, 
		\begin{align}
		    \vec{f}(t) = p\cev{\alpha}t^{2} + \left( ((1-\lambda)\vec{\alpha}-\lambda \cev{\alpha})(1-p)+p \right) t - \lambda(1-p),
		\end{align}
		whose coefficients may depend on $p$ and $\lambda$. Since it is concave up and $\vec{f}(0)=-\lambda(1-p)<0$, we must have $\vec{f}(1)>0$ in order to have a solution to $\vec{f}(t)=0$ in $(0,1)$. Hence 
		\begin{align}
			\vec{f}(1) 
			&= p(1+\cev{\alpha} -\vec{\alpha}+\lambda(2+\hat{\alpha}) ) + \vec{\alpha}-\lambda(2+\hat{\alpha})>0,
		\end{align} 
		which says that $p$ should be strictly greater than the first fraction in the definition of $F(p,\lambda)$ at \eqref{eq:F}. A symmetric argument for $\cev{q}$ shows that $p$ should dominate the other fraction in $F(p,\lambda)$. Hence we have shown that $\cev{q},\vec{q}<1$ implies $p>F(p,\lambda)$. Taking the contrapositive, we have $p\le F(p,\lambda)$ implies $\vec{q}=1$ or $\cev{q}=1$, which implies $\theta(p,\lambda) = 0$.
	\end{proof}

	\begin{remark} \thlabel{rem:aa}
		\normalfont
		{It follows from \eqref{eq:third} that $\vec q =1 = \cev q$ implies $\cev{\alpha}-\vec{\alpha} = (1-2\lambda)(p^{-1}-1)$. For atomless $\mu$, since $\cev{\alpha}+\vec{\alpha}=1$, we can obtain closed form solutions for $\cev{\alpha}$ and $\vec{\alpha}$. Then one can exactly solve the equation $p = F(p,\lambda,v)$ and obtain a conjectural formula for $p_{c}$, which would be valid in the regime where $\vec q =1 = \cev q$. However, mean-field heuristics from \cite{discrete} suggest that whenever $\lambda \neq 1/2$ or $v \neq 1$ it is only possible that exactly one of $\vec q$ or $\cev q$ can equal 1. Indeed, the density of what the authors call the ``minority species" decays at a summable rate, thus the expected number of that species to reach the origin is finite. Assuming this heuristic, it follows that $\vec q =1 = \cev q$ and atomless $\mu$ imply $\lambda=1/2$ and $\cev{\alpha}=\vec{\alpha}=1/2$, yielding $p_{c}=1/4$. On the other hand, when $\lambda=1/2$ and $v\ne 1$, we believe that $\cev{q}\ne \vec{q}$. So it seems that $\cev{q}=\vec{q}=1$ is only possible in the totally symmetric case. We emphasize that providing a conjectural closed form for $p_c(\lambda,v)$ is an open problem that, to our knowledge, is not contained in the physics literature.}
	\end{remark}

	

	\vspace{0.2cm}

	\section{Finite survival conditions and Proof of Theorem \ref{thm:main2}} \label{sec:theta_lemma}
	
	The goal of this section is to prove the following lemma, which characterizes the survival probability $\theta$ of a blockade from below using random variable $W$  (see \eqref{eq:W_def_1}) of finite support. 
	
	\begin{lemma} \label{lemma:theta}
		Fix $\mu,p, \lambda$, and $v$. For $\theta = \theta(p,\lambda,v)$ it holds that $$\displaystyle \theta = \max\left( 0, \, \sup_{k\ge 1}\, k^{-1}\E[W(1,k)] \right).$$
	\end{lemma}

	Before we proceed to prove Lemma \ref{lemma:theta}, we derive Theorem \ref{thm:main2} here. 
	
	\begin{proof}[{Proof of Theorem \ref{thm:main2}}]
		Some routine algebra lets us rewrite $F$ from \eqref{eq:F} as
		\begin{align}\label{eq:pc_1/4_3}
			F(p,\lambda)= \frac{1}{4+\hat{\alpha}}+\frac{2+\hat{\alpha}}{4+\hat{\alpha}}
			\max\left(\frac{\lambda(3+\hat{\alpha})-1-\vec{\alpha}}{\lambda(2+\hat{\alpha})+1+\cev{\alpha} - \vec{\alpha}},\,\frac{- (\lambda(3+\hat{\alpha})-1-\vec{\alpha})}{(1-\lambda)(2+\hat{\alpha})+1+\vec{\alpha}-\cev{\alpha}}  \right).
		\end{align}
		Note that one of the two terms in the maximum must be positive,  so this yields  
		\begin{align}
			F(p,\lambda,v) \ge \frac{1}{4+\hat{\alpha}}. \label{eq:unit1}
		\end{align}
		Notice that $\hat \alpha \leq 1$ and is zero in the atomless case. Combining this observation with \eqref{eq:unit1} and \eqref{eq:elementary} gives the claimed lower bound on $F$. {Then the claimed lower bound on $p_{c}^{-}$ follows from Theorem \ref{thm:main1}.}
		
		Lastly, we show the upper bound. Fix $\lambda$. Let $\cev{W}(a,b)$ and $\vec{W}(a,b)$ be as defined in \eqref{eq:W_def_1}. Then 
		\begin{align}
			\E[\cev{W}(1,1)] = p-\lambda(1-p) = p(1+\lambda)-\lambda
		\end{align}
		and 
		\begin{align}
			\E[\vec{W}(-1,-1)] = p-(1-\lambda)(1-p) = p(2-\lambda)-(1-\lambda).
		\end{align}
		Lemma~\ref{lem:vecW} then implies that 
		\begin{align}
			p > \max \left( \f{\lambda}{1 + \lambda}, \f{ 1 - \lambda}{2- \lambda}  \right) = f^*(\lambda) \implies \theta  >0.
		\end{align}
		Thus $p_{c}^{+}(\lambda) \leq f^*(\lambda)$.
	\end{proof}

	In the rest of this section, we prove Lemma~\ref{lemma:theta}. 
	We begin with a superadditivity property of $W$, which is proven in \cite[Lemma 15]{CBA}. For the sake of completeness, we state the lemma as the following proposition. 
	
	\begin{prop}[Lemma 15 in \cite{CBA}]\label{prop:superadditivity_Z}
		Let $a< b < c$ be integers. Then 
		\begin{align}
			W(a,c) \ge W(a,b) + W(b+1,c).
		\end{align}
	\end{prop}
	
	\begin{proof}
		The proof proceeds by considering the chain effect of the surviving left and right particle from $\mathbf{BA}[x_{a},x_{b}]$ and $\mathbf{BA}[x_{b+1},x_{c}]$ in the merged process $\mathbf{BA}[x_{a},x_{c}]$ and showing that the inequality is valid after each type of collision. This statement is proved in \cite[Lemma 15]{CBA} for an extended ballistic annihilation system where moving particles and survive collisions and collision between left and right particles may create a new blockade, and the proof does not use the assumption of symmetric particle velocity and density. Thus, our claim follows as a special case with virtually no modification. 
	\end{proof}
	
\begin{prop} \label{prop:new_sa}
    Suppose that $\theta >0$.  It holds that 
    \begin{align}
     \lim_{k \to \infty} \f{\dot N(1,k) - \dot N_{\mathbb R}(1,k)}{k} = 0 = \lim_{k \to \infty} \f{\dot N(-k,-1) - \dot N_{\mathbb R}(-k,-1)}{k} \label{eq:diff}
    \end{align}
    with $\dot N_{\mathbb R}(a,b)$ the number of blockades that survive in $[x_a,x_b]$ in the full process on $\mathbb R$ with all particles present. 
\end{prop}

\begin{proof}
Let $a < b$ be integers. Define $K_b=b$ if $\vec N(a,b) = 0$, and otherwise let $K_b$ be the smallest integer $k$ such that the leftmost surviving right particle in $\BA[x_a,x_b]$ is annihilated by $\b_k$ in $\BA[x_a,x_k]$. Similarly define $K_a$ as the index of the particle that destroys the rightmost surviving left particle from $\BA[x_a,x_b]$, with $K_a = a$ if no such left particle exists. The assumption that $\theta>0$ implies that all moving particles eventually are annihilated and thus $K_a$ and $K_b$ are almost surely finite. 

We claim that \eqref{eq:diff} follows from the inequality
    \begin{align}
    0 \leq \dot N(a,b) - \dot N_{\mathbb R}(a,b) \leq \cev N(K_b+1,\infty) + \vec N(-\infty, K_a-1)\label{eq:new_sa}.
    \end{align}
    {\color{black}Indeed, note that by the renewal that takes place between arriving moving particles, $\cev N(K_b+1,\infty)$ and  $\vec N(-\infty, K_a-1)$ are independent of $K_b$ and $K_a$, and they are distributed as geometric random variables with success parameters $1-\cev q$ and $1- \vec q$, respectively. As $\cev q$ and $\vec q$ are strictly less than $1$ whenever $\theta >0$, the geometric random variables are almost surely finite. Set $a=1$ and $b=k$. Then one has $\vec{N}(-\infty, K_{1}-1)/k\rightarrow 0$ almost surely since $\vec{N}(-\infty, K_{1}-1)$ is almost surely finite and does not depend on $k$. Also, $\P(\cev{N}(K_{k}+1,\infty)/k > k^{-1})$ is summable as $\cev{N}(K_{k}+1,\infty)$ has exponential tail, so by the Borel-Cantelli lemma, $\cev{N}(K_{k}+1,\infty)/k\rightarrow 0$ almost surely. The first part of \eqref{eq:diff} now follows from \eqref{eq:new_sa} with $a=1$ and $b=k$. The second part follows by the same reasoning with $a=-k$ and $b=-1$.}

    It remains to establish \eqref{eq:new_sa}. Call the left quantity $L:=\dot N(a,b) - \dot N_{\mathbb R}(a,b)$. First, we observe that $L$ counts the number of blockades contained in $I:=[x_a,x_b]$ that survive in the process restricted to $I$ but are destroyed in the full process on $\mathbb R$. With this viewpoint, it follows immediately that $L\geq0$. So, we turn our attention to bounding $L$ from above. Let $\mathscr I$ be the set of all indices of the blockades counted by $L$. Additionally let $\mathscr J$ be the set of indices of the particles counted by $R:={\color{black}\cev{N}}(K_b+1,\infty) + \vec N(-\infty, K_a-1)$. We claim that there is an injection $\pi \colon \mathscr I \to \mathscr J$. Thus, $L \leq R$, which is the statement in \eqref{eq:new_sa}. 
    
  Consider $i \in \mathscr I$. In the full process we must have that $\B_i$ is destroyed by a moving particle $\b_j$. Without loss of generality, we will assume that $\L_j$. If $j \in \mathscr J$, then set $\pi(i) = j$. On the other hand, if $j \notin \mathscr J$, then, by the definition of $\mathscr J$, the event $(K_b \not \leftarrow \L_j)_{(K_b,\infty)}$ occurs. Thus, $({\color{black}\b_i} \leftarrow \L_j)_{\mathbb R}$ occurs only if some $\R_{k}$ with $k \leq K_b$ unleashes $\L_j$ by destroying the blockade that $\L_j$ annihilated with in $\BA[x_{K_b+1},\infty)$. However, our construction is such that no right particles survive in $\BA[x_{K_a}, x_{K_b}]$. Thus, $\R_k$ must have similarly been unleashed by a left moving particle $\L_{j_1}$ from $I^c$. If $j_1 \in \mathscr J$, then set $\pi(i) = j_1$. Otherwise, iterate this reasoning to find a right moving particle $\R_{k_1}$ that was unleashed by some $\L_{j_2}$ and so on. As $\BA[x_{K_a}, x_{K_b}]$ contains no surviving moving particles, this chain must terminate at some $j_\ell \in \mathscr J$. We then set $\pi(i) = j_\ell$. As any given moving particle may unleash at most one other moving particle by destroying the blockade it annihilated with in either $I$ or $I^c$, we have $\pi$ is an injection, as desired. 
\end{proof}

	\begin{prop}\label{prop:theta_ergodic_avg}
		Suppose $\theta>0$. Then almost surely, 
		\begin{align}\lim_{k\rightarrow\infty} \frac{1}{k}\cev{N}(1,k)=\lim_{k\rightarrow\infty} \frac{1}{k}\cev{N}(-k,-1)=\lim_{k\rightarrow\infty} \frac{1}{k}\vec{N}(-k,-1)=\lim_{k\rightarrow\infty} \frac{1}{k}\vec{N}(1,k)=0, \label{eq:cevN}
		\end{align}
		and 
		\begin{align}
			&\lim_{k\rightarrow\infty} \frac{1}{k}\dot{N}(1,k)= \lim_{k\rightarrow\infty} \frac{1}{k}\dot{N}(-k,-1)=\theta. \label{eq:dotN}
		\end{align} 
		
	\end{prop}
	
	\begin{proof}
		Suppose $\theta>0$. We begin by proving the claimed equalities at \eqref{eq:cevN}. For integers $a<b$, define $\dot{N}_{\mathbb{R}}(a,b)=\sum_{a\le i \le b} \mathbf{1}((\B_{i} \text{ survives})_{(-\infty,\infty)})$. We let $\dot{N}_{\mathbb{R}}(a,\infty)=\lim_{b\rightarrow \infty} \dot{N}_{\mathbb{R}}(a,b)$. 
		First observe that for all $k \geq 1$ we have $$\cev{N}(1,k)=\cev{N}_{(0,\infty)}(1,k)\le \cev{N}_{(0,\infty)}(1,\infty).$$ By the renewal property of BA, we see that $\cev{N}(1,\infty)$ is a Geometric$(1-\cev{q})$ random variable. Note that $\cev{q}<1$ since $\theta>0$. Hence $\cev N(1,\infty)$ is almost surely finite. It follows that 
		\begin{align}
		\f{\cev{N}(1,k)}k \leq \f {\cev N(1,\infty)}{k} \rightarrow 0 \text{ almost surely}.\label{eq:geo}
		\end{align}
		
		  As for the limit involving $\cev N(-k,-1)$, by  translation invariance of the restricted process $\mathbf{BA}[x_{a},x_{b}]$, $\cev{N}(-k,-1)$ has the same distribution as $\cev{N}(1,k)$, which we just established is bounded by a Geometric$(1-\cev{q})$ random variable. Hence, $\cev{N}(-k,-1)$ is stochastically dominated by a Geometric$(1-\cev{q})$ random variable. It follows that 
		 $\P( \cev N(-k, -1) /k \geq \epsilon )$ is summable for any $\epsilon >0$. This observation along with the Borel-Cantelli lemma give
	       \begin{align}
	       \cev{N}(-k,-1)/k\rightarrow 0 \text{ almost surely.} \label{eq:bc}
	       \end{align}
	       A similar argument shows that $\vec{N}(-k,-1)/k\rightarrow 0$ and $\vec{N}(1,k)/k\rightarrow 0$ almost surely. This establishes the claims at \eqref{eq:cevN}.

		Next, we establish \eqref{eq:dotN}. Since the indicator for $\B_{i}$ surviving the full process is translation invariant, Birkhoff's ergodic theorem gives 
		\begin{align}
			\lim_{k\rightarrow\infty} \frac{1}{k}\dot{N}_{\mathbb{R}}(1,k)=\lim_{k\rightarrow\infty} \frac{1}{k}\dot{N}_{\mathbb{R}}(-k,-1)=\theta.
		\end{align}
		Proposition~\ref{prop:new_sa} implies that the above limit is also equal to $\lim k^{-1} \dot N(1,k)$ and $\lim k^{-1} \dot N(-k,-1)$, which \eqref{eq:dotN}.
		\end{proof}

	In what follows we fix $\mu,p, \lambda$, and $v$. Recall that $\theta = (1- \cev q)(1-\vec q)$. We will characterize  $\theta$ in terms of the quantity
	\begin{align}
		\theta_0 := \max \left( 0 , \sup_{k\geq 1} \,\, k^{-1} \E[ W(1,k) ] \right) \label{eq:theta0}. 
	\end{align}

	Finally, we derive Lemma \ref{lemma:theta}.
	
	\begin{proof}[Proof of Lemma \ref{lemma:theta}]
	    We first show $\theta\le \theta_{0}$. This inequality is trivial if $\theta=0$, so we may assume $\theta >0$. By Proposition \ref{prop:theta_ergodic_avg}, a trivial rephrasing, and then Fatou's lemma we have
		\begin{align} 
			\theta = \lim_{k \to \infty} k^{-1} W(1,k) &= \E \left[ \liminf_{k\rightarrow \infty}\,\, k^{-1}W(1,k) \right] \le \liminf_{k\rightarrow \infty} \,\,k^{-1} \E[W(1,k)] \leq \theta_{0}.
		\end{align} 
	
		It remains to show $\theta \ge \theta_{0}$. Assume without loss of generality that $\theta_{0}>0$. We first claim that this implies $\theta>0$. Denote $\cev{\theta}=(1-\cev{q})$ and $\vec{\theta}=(1-\vec{q})$ so that $\theta = \cev{\theta}\vec{\theta}$. As $\theta_0>0$, we may fix an integer $k$ with $\E[W(1,k)]>0$. Denote $K_{m}:=km$ for $m=0,1,\dots$.  Observe that $S_{n}:=\sum_{m=0}^{n-1}W(K_{m}+1,K_{m+1})$, $n\ge 1$ is a random walk with positive drift $\E[W(1,k)]>0$. Hence $S_{n}-S_{1}$ for $n\ge 1$ has a positive probability to stay above $0$. Independently with probability $p^{k}>0$, $S_{1}=W(1,K_{1})=k$, meaning that all $k$ particles in the first interval $[x_{1},x_{K_{1}}]=[x_{1},x_{k}]$ are blockades. Hence, with a positive probability, we have $S_{1}=k$ and $S_{n}\ge k+1$ for all $n\ge 2$. 
		
		Now note that by Proposition \ref{prop:superadditivity_Z} we have 
		\begin{align}
			W(1,K_{n})\ge S_{n} \qquad \forall n\ge 1. \label{eq:Klb}
		\end{align}
		Thus we have that $\P(A)>0$, where $A$ denotes the event that $W(1,K_{1})=k$ and $W(1,K_{n})\ge k+1$ for all $n\ge 1$. We claim that, on this event, the origin is not visited from right, which implies $\cev{\theta}>0$. Indeed,  $\textbf{BA}[1,K_{n}]$ has at least $k+1$ surviving blockades since $W(1,K_{n})\ge k+1$, and there are at most $k$ surviving left arrows in $\textbf{BA}[K_{n}+1,K_{n+1}]$ (since $K_{n+1}-K_{n}=k$), the origin is never visited by a left arrow in $\textbf{BA}[1,K_{n+1}]$. Continuing by induction, we see that $A$ implies $\cap_{n=1}^{\infty}(\B_{1}\nleftarrow \L)_{[x_{1}, x_{K_{n}}]}$. By continuity of measure, it follows that $\cev{\theta}=\lim_{n\rightarrow \infty} \P((\B_{1}\nleftarrow \L)_{[x_{1}, x_{K_{n}}]}\,|\, A)\, \P(A)>0$, as desired. A symmetric argument shows $\vec{\theta}>0$. (For this, one may use the fact that the distribution of $W(a+n, b+n)$ does not depend on $n$, so $\E[W(1,k)]>0$ implies $\E[W(-k,-1)]>0$.) Hence $\theta = \cev{\theta}\vec{\theta}>0$, proving the claim. 
		
		
		Now, fix $\delta \in (0,1)$ such that $\theta_{0}>\delta$. We wish to show $\theta>\delta$. Fix an integer $k\ge 1$ such that $k^{-1}\E[W(1,k)]>\delta$. As we previously established that $\theta>0$, we may apply Proposition \ref{prop:theta_ergodic_avg}, \eqref{eq:Klb}, and the strong law of large numbers for $S_{n}$, to conclude that
		\begin{align}\label{eq:pf_theta_formula}
			\theta = \lim_{n\rightarrow\infty} \frac{1}{n}W(1,n) = \liminf_{n\rightarrow\infty} \frac{n}{K_{n}}\frac{1}{n}W(1,K_{n})\ge \liminf_{n\rightarrow\infty} \frac{n}{K_{n}}\frac{1}{n}S_{n} = \f 1k \E[W(1,k)]>\delta.
		\end{align}
		This shows the assertion.
	\end{proof}
	
\begin{remark} \label{rem:theta}
    If we had worked with $\cev W$ instead of $W$, then the indices $K_n$ would be random. This is because $\cev W$ is only suparadditive when combining intervals for which the leftmost interval has no surviving right particles. Thus, the analogues of our $K_n$ are found by extending the configuration until all such right particles are destroyed. Working with $W$, which enjoys the more general superadditivity at Proposition \ref{prop:superadditivity_Z}, lets us avoid any analysis of these random extensions, which is helpful for proving Lemma \ref{lemma:theta}. Specifically, see \eqref{eq:pf_theta_formula} where the equality $(n/K_n)(1/n) = 1/k$ which would be nontrivial to deduce if $K_n$ were random. Indeed, one would need a better understanding of the distance moving particles travel before annihilation. This quantity presumably has an exponential tail, but a proof of this feature is not immediately obvious.  
\end{remark}
	
Next, we obtain our upper bound $p_c(\lambda,v) \leq f^*(\lambda)$ from Theorem~\ref{thm:main2} by applying a weaker if and only if condition for fixation to occur. First we prove a dichotomy concerning moving particle survival.
	
	\begin{prop}\label{lemma:alpha_half_implications}
		The following implications hold:
		\begin{align}
			p\cev{q}\vec{q}(\cev{\alpha}-\vec{\alpha})\le (1-2\lambda)(1-p) & \Longrightarrow  \text{$\vec{\beta}=0$ and $\vec{q}\le \cev{q}$},\\
			p\cev{q}\vec{q}(\cev{\alpha}-\vec{\alpha})\ge (1-2\lambda)(1-p) & \Longrightarrow  \text{$\cev{\beta}=0$ and $\vec{q}\ge \cev{q}$}.
		\end{align}
		In particular, either \textup{[}$\vec{\beta}=0$ and $\vec{q}\le \cev{q}$\textup{]} or \textup{[}$\cev{\beta}=0$ and $\vec{q}\ge \cev{q}$\textup{]} hold.
	\end{prop}

	\begin{proof}	
		This argument relies on \eqref{eq:third}, which we restate for easy reference,
		\begin{align}
			p\vec{q}\cev{q}(\cev{\alpha}-\vec{\alpha}) + p(\cev{q}-\vec{q})-(1-2\lambda)(1-p) =  \vec{\beta}-\cev{\beta}.
		\end{align}
		Suppose $\vec{\beta}>0$. Note that this implies $\vec{q}=1$ and $\cev{\beta}=0$. Hence the last equation in Lemma \ref{lemma:main_recursion} yields 
		\begin{align}
			0<\vec{\beta} &= p\vec{q}\cev{q}(\cev{\alpha}-\vec{\alpha}) - (1-2\lambda)(1-p)+p(\cev{q}-1)\ \\
			&\le p\vec{q}\cev{q}(\cev{\alpha}-\vec{\alpha}) - (1-2\lambda)(1-p).
		\end{align}
		Hence $p\cev{q}\vec{q}(\cev{\alpha}-\vec{\alpha})> (1-2\lambda)(1-p)$. Then taking the contrapositive shows that $p\cev{q}\vec{q}(\cev{\alpha}-\vec{\alpha})\le (1-2\lambda)(1-p)$ implies $\vec{\beta}=0$.

		Next, suppose $p\cev{q}\vec{q}(\cev{\alpha}-\vec{\alpha})\le (1-2\lambda)(1-p)$. We just proved that this implies $\vec{\beta} = 0$, and it remains to show that $\vec{q}\le \cev{q}$. By the last \eqref{eq:third} and since $\vec{\beta}=0$, we get 
		\begin{align}
			\cev{\beta} &= -p\vec{q}\cev{q}(\cev{\alpha}-\vec{\alpha}) + p(\vec{q}-\cev{q})+(1-2\lambda)(1-p) \ge p(\vec{q}-\cev{q}).
		\end{align}
		If $\vec{q}>\cev{q}$, then the above inequality yields $\cev{\beta}>0$, which then implies $\cev{q}=1$. But since $\vec{q}>\cev{q}$, this is a contradiction. Thus we must have $\vec{q}\le \cev{q}$. This shows the first implication in the assertion. A symmetric argument shows the second.   
	\end{proof}
	

    \begin{lemma}\label{lem:vecW}
       The following holds: 
        \begin{align}
        \theta>0 \quad \text{ if and only if } \quad \exists k,\ell \,\, s.t. \,\, \text{$\E[ \cev{W}(1,k) ]>0$ and  $\E[\vec{W}(-\ell, -1)]>0$}.
    \end{align}
    \end{lemma}

\begin{proof}
    First suppose $\theta>0$. According to Lemma \ref{lemma:theta}, we have $\E[W(1,k)]>0$ for some $k\ge 1$. Since $W(1,k) \ge \max(\cev{W}(1,k),\vec{W}(1,k))$ and $\vec{W}(1,k)$, we have $\E[\cev{W}(1,k)]>0$ and $\E[\vec{W}(1,k)]>0$. By translation invariance of the restricted process $\mathbf{BA}[x_{a},x_{b}]$, $\cev{W}(1,k)$ and $\vec{W}(-k,-1)$ have the same distribution. Hence  $\E[\vec{W}(-k,-1)] = \E[\vec{W}(1,k)]>0$. This shows the forward implication.

    Next, suppose that there exists  $k,\ell\ge 1$ such that $\E[ \cev{W}(1,k) ]>0$ and  $\E[\vec{Z}(-\ell, -1)]>0$. Denote $\cev{\theta}=(1-\cev{q})$ and $\vec{\theta}=(1-\vec{q})$ so that $\theta = \cev{\theta}\vec{\theta}$. By Proposition \ref{lemma:alpha_half_implications}, we may assume that $\vec{\beta}=0$ and $\vec{q}\le \cev{q}$, so that $\vec{\theta}\ge \cev{\theta}$. This implies $\theta = \cev{\theta}\vec{\theta} \ge \cev{\theta}^{2}$. Thus it suffices to show that $\cev{\theta}>0$ when $\vec{\beta}=0$ and $\E[\cev{W}(1,k)]>0$ hold for some $k$.

	We claim that under the assumption $\vec{\beta}=0$ there exists a sequence of almost surely finite indices $0=K_{0}<K_{1}<K_{2}<\cdots$ such that the random variables $\cev{W}(K_{m}+1,K_{m+1})$ (indexed by $m$) are i.i.d.\ with $\E[\cev{W}(K_{m}+1,K_{m+1})] = \E[\cev{W}(1,k)]>\delta_{1}$ and $\vec{N}(K_{m}+1, K_{m+1})=0$. Assuming this claim, we note that $S_{n}:=\sum_{m=0}^{n-1}\cev{W}(K_{m}+1,K_{m+1})$ is a random walk with positive drift $\E[\cev{W}(1,k)]>0$. Also, $\cev{W}$ is superadditive in the sense that $\cev{W}(a,c) \ge \cev{W}(a,b)+\cev{W}(b,c)$ for $a<b<c$ whenever $\vec{N}(a,b)=0$ (see \cite[Lemma 12]{HST}). Hence we have 
	\begin{align}
	\cev{W}(1,K_{n})\ge S_{n} \qquad \forall n\ge 1.
	\end{align}
	Hence by the strong law of large numbers, almost surely, 
	\begin{align}
	\liminf_{n\rightarrow\infty} K_{n}^{-1} \cev{W}(1,K_{n}) \ge \E[\cev{W}(1,k)]>0.
	\end{align} 
	As in the proof of Lemma \ref{lemma:theta}, this implies $\cev{\theta}>0$. 
	
	Now we justify the claim in the previous paragraph concerning the existence of $K_0,K_1,\hdots$. We first consider $\mathbf{BA}[x_{1},x_{k}]$. Let $K_{1}=k$ if $\vec{N}(1,k)=0$. Otherwise, let $\tau_{1}$ be the index of the leftmost such right particle surviving in $\mathbf{BA}[x_{1},x_{k}]$. Define $K_{1}\ge 1$ to be the unique index such that  $(\R_{\tau_{1}}\rightarrow \bullet_{K_{1}})_{[x_{1},x_{K_{1}}]}$ occurs. The value $K_{1}$ exists almost surely since $\vec{\beta}=0$. Furthermore, we have that $\E[\vec{Z}(1,K_{1})] = \E[\vec{Z}(1,k)]>\delta$ since all particles in the interval $[x_{\tau_{1}},x_{K_{1}}]$ get annihilated in $\mathbf{BA}[x_{1},x_{k}]$. This also yields $\vec{N}(x_{1},x_{K_{1}})=0$. Note that $\mathbf{BA}[x_{K_{1}+1},\infty)$ is independent and has the same law as $\mathbf{BA}[x_{1},\infty)$. Iterating the same procedure starting from $\mathbf{BA}[x_{K_{1}+1},x_{K_{1}+k}]$, we can find a finite index $K_{2}>K_{1}$ such that $\cev{W}(K_{0}+1,K_{1})$ and $\cev{W}(K_{1}+1,K_{2})$ are i.i.d. Repeating this procedure and using translation invariance, yields the claimed sequence.
\end{proof}

	
	\section{Continuity of the main quantities}
	\label{sec:continuity}

	In this section, we prove Theorem \ref{thm:continuity}. Recall that if $\mathfrak{X}$ is a topological space, we say a function  $f:\mathfrak{X}\rightarrow \mathbb{R}$ is \textit{upper semi-continuous} (USC) (resp., \text{lower semi-continuous} (LSC)) if $\{x\in \mathfrak{X}\,|\, f(x)<u \}$ (resp., $\{x\in \mathfrak{X}\,|\, f(x)>u \}$)is open in $\mathfrak{X}$ for all $u\in \mathbb{R}$. If $\mathfrak{X}$ is a metric space, then $f$ is USC (resp., LSC) if and only if there exists a sequence of continuous functions $f_{k}:\mathfrak{X}\rightarrow \mathbb{R}$ such that $f_{k}(x)\searrow f(x)$ (resp., $f_{k}(x)\nearrow f(x)$) for all $x\in \mathfrak{X}$ as $k\rightarrow \infty$.

	We first show semi-continuity of the main quantities by a standard finite approximation argument. 
	
	\begin{prop}\label{prop:LSC}
		Fix arbitrary spacing distributions and parameters $(p,\lambda)\in (0,1)^{2}$.
		\begin{description}
			\item[(i)] The following quantities are LSC in $p$ and $\lambda$: $\cev{q}$, $\vec{q}$, $\cev{\alpha}\cev{q}\vec{q}$,  $\vec{\alpha}\cev{q}\vec{q}$, $\hat{\alpha}\cev{q}\vec{q}$, \\
			$\P(\R_{1}\rightarrow \B)$,  and $\P(\R_{1}\rightarrow \B)$.
			
			\vspace{0.1cm}
			\item[(ii)] The following quantities are USC in $p$ and $\lambda$: $\theta$, $\cev{\beta}$, and $\vec{\beta}$.
		\end{description}
	\end{prop} 
	
	\begin{proof}
		We first show that $\cev{q}$ and $\vec{q}$ are LSC. Let $\cev{\sigma}$ be the index of the particle that first reaches $0$ from the right, and the reverse for $\vec{\sigma}$. We may set $\cev{\sigma}=\infty$ and $\vec{\sigma}=-\infty$ if there is no such particle, respectively. Note that $\cev{q}=\P(\cev{\sigma}<\infty) = \P(\cev{\tau}<\infty)$, and likewise for $\vec{q}$ and $\vec{\sigma}$. Then we may write  
		\begin{align}
			\cev{q} &=  \sum_{n=1}^{\infty} \P(\cev{\sigma}=n).
		\end{align}
		Since the event that $\cev{\sigma}=n$ is completely determined by the first $n$ particles to the right of the origin, by conditioning on their location and velocity, we see that $ \P(\cev{\sigma}=n)$ is a polynomial in $p$ and $\lambda$. Since $\P(\cev{\sigma}\le n)= \sum_{k=1}^{n}\P(\cev{\sigma}=k)\nearrow \cev{q}$, this shows that $\cev{q}$ is LSC in $p$ and $\lambda$. A similar argument shows the lower semi-continuity of $\vec{q}$, 	$\P(\R_{1}\rightarrow \B)$,  and $\P(\R_{1}\rightarrow \B)$.

		Next, we show the lower semi-continuity of $\cev{\alpha}\cev{q}\vec{q}$. By conditioning on the values of $\vec{\sigma}$, we obtain
		\begin{align}
			\cev{\alpha} \cev{q}\vec{q}=\P(\cev{\tau}\le \vec{\tau},\, \vec{\tau}<\infty) &= \sum_{n=1}^{\infty} \P(\cev{\tau} \le |x_{-n}|\,|\, \vec{\sigma}=n) \P(\vec{\sigma}=n) \\
			&= \sum_{n=1}^{\infty} \P( (0\leftarrow \L)_{[0,v|x_{-n}|]} ) \P(\vec{\sigma}=n).
			\label{eq:alpha_k}
		\end{align}    
		Note that the summands in (\ref{eq:alpha_k}) are continuous in $p$ and $\lambda$. Hence if we let $(\cev{\alpha} \cev{q}\vec{q})_{k}$ denote the $k$th partial sum, then $(\cev{\alpha} \cev{q}\vec{q})_{k}\nearrow \cev{\alpha} \cev{q}\vec{q}$. A similar argument holds for $\vec{\alpha} \cev{q}\vec{q}$. A similar argument shows the lower semi-continuity of $\vec{\alpha}\cev{q}\vec{q}$ and $\hat{\alpha}\cev{q}\vec{q}$. This shows (i). 
		
		Lastly, we show (ii). The upper-semicontinuity of $\theta$ follows from the lower semi-continuity of $\cev{q}$ and $\vec{q}$ we have just shown and the relation $\theta = (1-\cev{q})(1-\vec{q})$. For $\vec{\beta}$, we first write 
		\begin{align}
			\vec{\beta} = \lambda(1-p) - \P(\R_{1}\rightarrow \bullet).
		\end{align} 
		Note that $\P((\R_{1}\rightarrow \bullet_{k})_{[x_{1},x_{k}]} )\nearrow \P(\R_{1}\rightarrow \bullet)$ as $k\rightarrow \infty$. Since the event $(\R_{1}\rightarrow \bullet_{k})_{[x_{1},x_{k}]}$ depends only on the first $k$ particles, its probability is a polynomial in $p$ and $\lambda$. Hence $\P(\R_{1}\rightarrow \bullet)$ is LSC, and the above equation yields that $\vec{\beta}$ is USC. A similar argument applies to $\cev{\beta}$.
	\end{proof}
	
	Next, we show continuity of $\theta$ in all cases. 
	
	\begin{prop}\label{prop:continuity_theta}
		$\theta=\theta(p,\lambda,v)$ is continuous in $p$ and $\lambda$. 
	\end{prop}

	\begin{proof}
		Since the maximum of two LSC functions is LSC, Lemma \ref{lemma:theta} implies that $\theta$ is the maximum of 0, which is LSC, and the minimum of two suprema of finite polynomials in $p$ and $\lambda$, which is again LSC. Thus, $\theta$ is LSC in both $p$ and $\lambda$. On the other hand, $\theta = (1- \cev q)(1- \vec q)$ and both $\cev{q}$ and $\vec{q}$ are LSC by Proposition \ref{prop:LSC}. Hence $1-\cev{q}$ and $1-\vec{q}$ are USC, and since they are both non-negative, their product is also USC. 
	\end{proof}
	
	Our next aim is the continuity of $\cev{\beta}$ and $\vec{\beta}$. In the following proposition, we first deduce the continuity $\cev{\beta}-\vec{\beta}$ from that of $\theta$ by mass transport principle.
	
	\begin{prop}\label{prop:MTP_continuity}
		The followings hold:
		\begin{description}
			\item[(i)] $p(1-\theta) = \P(\R_{1}\rightarrow \B) + \P(\B\leftarrow \L_{1})$.
			
			\vspace{0.1cm}
			\item[(ii)] $\P(\R_{1}\rightarrow \B)$,  $\P(\B\leftarrow \L_{1})$, and $\cev{\beta}-\vec{\beta}$ are continuous in $p$ and $\lambda$.

		\end{description}
	\end{prop}
	
	\begin{proof}
		For $a,b\in \mathbb{Z}$, define the indicator variable
		$Z(a,b) = \1\left( (\R_{a}\rightarrow \B_{b})\lor (\B_{b}\leftarrow \L_{a}) \right).$
		Note that 
		\begin{align}
			\E \left[ \sum_{b\in \mathbb{Z}} Z(b,1) \right] =  \P(\text{$\B_{1}$ is annihilated in the full process}) = p(1-\theta).
		\end{align}
		By the mass transport principle, the above equals  
		\begin{align}
			\E \left[ \sum_{b\in \mathbb{Z}} Z(1,b) \right] = \P(\R_{1}\rightarrow \B) + \P(\B\leftarrow \L_{1}). 
		\end{align}
		This shows (i). 
		
		Next, we show (ii). We first note that both 	$\P(\R_{1}\rightarrow \B)$  and $\P(\R_{1}\rightarrow \B)$ are  continuous in $p$ and $\lambda$. This follows from (i) and  continuity of $\theta$ (Proposition \ref{prop:continuity_theta}) and the lower semi-continuity of both $\P(\R_{1}\rightarrow \B)$  and $\P(\R_{1}\rightarrow \B)$ (Proposition \ref{prop:LSC}). Then note that 
		\begin{align}
			\cev{\beta} - \vec{\beta} &= (1-\P\left( \R_{1}\rightarrow \bullet \right)) - (1-\P\left( \bullet \leftarrow\L_{1}\right)) \\
			&= \P\left( \B \leftarrow\L_{1}\right) + \P\left( \R \leftarrow\L_{1}\right)- \P\left( \R_{1}\rightarrow \B \right) - \P\left( \R_{1}\rightarrow \L \right) \\
			&= \P\left( \B \leftarrow\L_{1}\right) - \P\left( \R_{1}\rightarrow \B \right),
		\end{align}
		where the second equality follows from the last equation of Proposition \ref{prop:p_pm}. Since both $\P(\R_{1}\rightarrow \B)$  and $\P(\R_{1}\rightarrow \B)$ are  continuous in $p$ and $\lambda$, this shows the continuity of $\cev{\beta}-\vec{\beta}$. 
	\end{proof}
	
	The following ``pasting lemma'' is standard in topology, see \cite{munkres1974topology}. 
	
	\begin{lemma}[Pasting lemma]\label{lemma:pasting}
		Let $A,B$ be topological spaces and let $f:A\rightarrow B$ be a function. Suppose $A=X\cup Y$, where both $X$ and $Y$ are either open or closed. If $f$ is continuous when restricted to both $X$ and $Y$, then $f$ is continuous. 
	\end{lemma}
	
	\begin{proof}
		Suppose both $X$ and $Y$ are open. Let $U\subseteq B$ be open. Note that 
		\begin{align}
			f^{-1}(U) = \left( f^{-1}(U)\cap X \right)\cup \left( f^{-1}(U)\cap Y \right) = (f|_{X})^{-1}(U) \cup (f|_{Y})^{-1}(U).
		\end{align}
		The last expression is open by the hypothesis. This shows $f$ is continuous. The case when $X$ and $Y$ are both closed can be argued similarly. 
	\end{proof}
	
	Now we show the continuity of $\cev{\beta}$ and $\vec{\beta}$. 
	
	\begin{prop}\label{prop:continuity_beta}
		$\cev{\beta}$ and $\vec{\beta}$ are continuous in $p$ and $\lambda$. 
	\end{prop}
	
	\begin{proof}
		Recall that $\cev{\beta},\vec{\beta}\in [0,1]$ and  $\cev{\beta}\vec{\beta}=0$ in all cases. Hence we have 
		\begin{align}
			\{ \vec{\beta}=0 \} = \{ \cev{\beta}-\vec{\beta}\ge 0 \}.
		\end{align}
		Then since $\cev{\beta}-\vec{\beta}$ is continuous by Proposition \ref{prop:MTP_continuity}, the right hand side is closed being the inverse image of a closed set under a continuous function. So this  shows $\{\vec{\beta}=0\}$ is closed. A similar argument shows that $\{\cev{\beta}=0\}$ is closed. Since $\cev{\beta}\vec{\beta}=0$, we can write $(0,1)^{2}=\{\cev{\beta}=0\}\cup \{\vec{\beta}=0\}$. 
		Since these subsets are closed, by Lemma \ref{lemma:pasting}, it suffices to show that $\cev \beta$ and $\vec \beta$ are continuous when restricted to $\{ \vec{\beta}=0 \}$ and $\{ \cev{\beta}=0 \}$. On the regime where $\cev{\beta}=0$, $\vec{\beta}=\vec{\beta}-\cev{\beta}$ is continuous by Proposition \ref{prop:MTP_continuity}. On the regime where $\vec{\beta}=0$, $\vec{\beta}$ is continuous being a constant function. This shows the continuity of $\vec{\beta}$. A similar argument shows the continuity of $\cev{\beta}$. 
	\end{proof}
	
	
	Finally, we prove Theorem \ref{thm:continuity}.
	
	\begin{proof}[Proof of Theorem \ref{thm:continuity}]
		We have shown continuity of $\theta$, $\cev{\beta}$, and $\vec{\beta}$ in Propositions \ref{prop:continuity_theta} and \ref{prop:continuity_beta}. In order to show continuity of $\cev{q}\vec{q}\hat{\alpha}$, recall \eqref{eq:RB} and its symmetric counterpart:
		\begin{align}\label{eq:arrow-block-collision-prob}
			\P\left( (\R_{1} \rightarrow \B) \right) &= p\vec{q} (1-\cev{q}) + p\vec{q}\cev{q}\vec{\alpha} \\
			\P\left( (\B \leftarrow \L_{1}) \right) &= p\cev{q} (1-\vec{q}) + p\vec{q}\cev{q}\cev{\alpha}.
		\end{align}
		By Proposition \ref{prop:MTP_continuity}, the left-hand-sides are continuous in all cases. Hence by adding the two equations and using the relation $\cev{\alpha}+\vec{\alpha}=1+\hat{\alpha}$,
		\begin{align}
			\vec{q}+\cev{q}-2\cev{q}\vec{q} + \cev{q}\vec{q}(\cev{\alpha}+\vec{\alpha}) = 1-(1-\cev{q})(1-\vec{q}) + \cev{q}\vec{q}\hat{\alpha} = \textup{continuous in $p$ and $\lambda$}.
		\end{align}
		Since $\theta=(1-\cev{q})(1-\vec{q})$ is continuous by Proposition \ref{prop:continuity_theta}, it follows that $\cev{q}\vec{q}\hat{\alpha}$ is continuous.

		Next, assume $\theta>0$. Then $\cev{q},\vec{q}<1$, so we can write $1-\cev{q} = \theta/(1-\vec{q}) $. Since $\theta$ is continuous and $1-\vec{q}$ is positive and USC, it follows that $1-\cev{q}$ is LSC. This implies that $\cev{q}$ is USC. But since $\cev{q}$ is LSC by Proposition \ref{prop:LSC}, this shows that $\cev{q}$ is continuous. A similar argument shows $\vec{q}$ is continuous.  Recall that the left-hand-side of \eqref{eq:arrow-block-collision-prob} is continuous in all cases. Since $\cev{q}$ and $\vec{q}$ are continuous on $\theta>0$, it follows that $\vec{q}\cev{q}\vec{\alpha}$ is also continuous on $\theta>0$. A similar argument shows that $\vec{q}\cev{q}\cev{\alpha}$ is continuous on $\theta>0$. Note that that $\cev q, \vec q >0$ and that $f/g$ is continuous whenever $f$ and $g$ are continuous with $g\neq 0$. Hence by writing
		\begin{align}
			\cev{\alpha} = (\cev{\alpha}\cev{q}\vec{q})/\cev{q}\vec{q}, \qquad \vec{\alpha} = (\vec{\alpha}\cev{q}\vec{q})/\cev{q}\vec{q}, \qquad \hat{\alpha} = (\hat{\alpha}\cev{q}\vec{q})/\cev{q}\vec{q},
		\end{align}
		we see that $\cev{\alpha}$, $\vec{\alpha}$, and $\hat{\alpha}$ are continuous. 
	\end{proof}

	\section*{Acknowledgments}
	The authors appreciate valuable comments from Rick Durrett and Tom Liggett. Junge was partially supported by NSF Grant DMS \#1855516 and Lyu was partially supported by NSF Grant DMS \#2010035.

	\bibliographystyle{amsalpha}
	\bibliography{BA}

\newcommand{\etalchar}[1]{$^{#1}$}
\providecommand{\bysame}{\leavevmode\hbox to3em{\hrulefill}\thinspace}
\providecommand{\MR}{\relax\ifhmode\unskip\space\fi MR }
\providecommand{\MRhref}[2]{%
  \href{http://www.ams.org/mathscinet-getitem?mr=#1}{#2}
}
\providecommand{\href}[2]{#2}
\begin{thebibliography}{BNRL93b}

\bibitem[BGJ18]{BGJ}
Debbie {Burdinski}, Shrey {Gupta}, and Matthew {Junge}, \emph{{The upper
  threshold in ballistic annihilation}}, (2018, arXiv ID: 1805.10969.

\bibitem[BJL{\etalchar{+}}20]{CBA}
Luis Benitez, Matthew Junge, Hanbaek Lyu, Maximus Redman, and Lily Reeves,
  \emph{Three-velocity coalescing ballistic annihilation}, 2020, arXiv ID:
  2010.15855.

\bibitem[BM20]{bullets2}
Nicolas Broutin and Jean-Fran{\c{c}}ois Marckert, \emph{The combinatorics of
  the colliding bullets}, Random Structures \& Algorithms \textbf{56} (2020),
  no.~2, 401--431.

\bibitem[BNRL93a]{continuous}
E~Ben-Naim, S~Redner, and F~Leyvraz, \emph{Decay kinetics of ballistic
  annihilation}, Physical review letters \textbf{70} (1993), no.~12, 1890.

\bibitem[BNRL93b]{b2}
\bysame, \emph{Decay kinetics of ballistic annihilation}, Physical review
  letters \textbf{70} (1993), no.~12, 1890.

\bibitem[CPY90]{b3}
G.~F. Carnevale, Y.~Pomeau, and W.~R. Young, \emph{Statistics of ballistic
  agglomeration}, Physical Review Letters \textbf{64} (1990), no.~24,
  2913--2916.

\bibitem[CRS18]{CRS}
M.~Cabezas, L.~T. Rolla, and V.~Sidoravicius, \emph{Recurrence and density
  decay for diffusion-limited annihilating systems}, Probab. Theory Related
  Fields \textbf{170} (2018), no.~3-4, 587--615. \MR{3773795}

\bibitem[DGJ{\etalchar{+}}19]{parking}
Michael Damron, Janko Gravner, Matthew Junge, Hanbaek Lyu, and David Sivakoff,
  \emph{Parking on transitive unimodular graphs}, The Annals of Applied
  Probability \textbf{29} (2019), no.~4, 2089--2113.

\bibitem[DKJ{\etalchar{+}}19]{bullets}
Brittany Dygert, Christoph Kinzel, Matthew Junge, Annie Raymond, Erik Slivken,
  Jennifer Zhu, et~al., \emph{The bullet problem with discrete speeds},
  Electronic Communications in Probability \textbf{24} (2019).

\bibitem[DLS20]{damron2020stretched}
Michael Damron, Hanbaek Lyu, and David Sivakoff, \emph{Stretched exponential
  decay for subcritical parking times on $\mathbb{Z}^{d}$}, arXiv preprint
  arXiv:2008.05072 (2020).

\bibitem[DRFP95]{b4}
Michel Droz, Pierre-Antoine Rey, Laurent Frachebourg, and Jarosław Piasecki,
  \emph{Ballistic-annihilation kinetics for a multivelocity one-dimensional
  ideal gas}, Physical Review \textbf{51} (1995), no.~6, 5541--5548 (eng).

\bibitem[EF85a]{2speeds}
Yves Elskens and Harry~L. Frisch, \emph{Annihilation kinetics in the
  one-dimensional ideal gas}, Phys. Rev. A \textbf{31} (1985), 3812--3816.

\bibitem[EF85b]{b5}
\bysame, \emph{Annihilation kinetics in the one-dimensional ideal gas}, Phys.
  Rev. A \textbf{31} (1985), 3812--3816.

\bibitem[FL18]{foxall2018clustering}
Eric Foxall and Hanbaek Lyu, \emph{Clustering in the three and four color
  cyclic particle systems in one dimension}, Journal of Statistical Physics
  \textbf{171} (2018), no.~3, 470--483.

\bibitem[HST21]{HST}
John Haslegrave, Vladas Sidoravicius, and Laurent Tournier, \emph{Three-speed
  ballistic annihilation: phase transition and universality}, Selecta
  Mathematica \textbf{27} (2021), no.~5, 1--38.

\bibitem[HT20]{HT}
John Haslegrave and Laurent Tournier, \emph{Combinatorial universality in
  three-speed ballistic annihilation}, 2020, arXiv ID: 2004.09119.

\bibitem[KR84]{recombination}
K.~Kang and S.~Redner, \emph{Scaling approach for the kinetics of recombination
  processes}, Phys. Rev. Lett. \textbf{52} (1984), 955--958.

\bibitem[KRL95a]{discrete}
PL~Krapivsky, S~Redner, and F~Leyvraz, \emph{Ballistic annihilation kinetics:
  The case of discrete velocity distributions}, Physical Review E \textbf{51}
  (1995), no.~5, 3977.

\bibitem[KRL95b]{b8}
\bysame, \emph{Ballistic annihilation kinetics: The case of discrete velocity
  distributions}, Physical Review E \textbf{51} (1995), no.~5, 3977.

\bibitem[KS01]{b9}
PL~Krapivsky and Cl{\'e}ment Sire, \emph{Ballistic annihilation with continuous
  isotropic initial velocity distribution}, Physical review letters \textbf{86}
  (2001), no.~12, 2494.

\bibitem[MS84]{dla_fractal}
P~Meakin and H~E Stanley, \emph{Novel dimension-independent behaviour for
  diffusive annihilation on percolation fractals}, Journal of Physics A:
  Mathematical and General \textbf{17} (1984), no.~4, L173.

\bibitem[Mun74]{munkres1974topology}
James~R Munkres, \emph{Topology; a first course [by] james r. munkres},
  Prentice-Hall, 1974.

\bibitem[Pia95]{b12}
Jaros{\l}aw Piasecki, \emph{Ballistic annihilation in a one-dimensional fluid},
  Physical Review E \textbf{51} (1995), no.~6, 5535.

\bibitem[ST17]{ST}
Vladas Sidoravicius and Laurent Tournier, \emph{Note on a one-dimensional
  system of annihilating particles}, Electron. Commun. Probab. \textbf{22}
  (2017), 9 pp.

\bibitem[TEW98]{b7}
Balint Toth, Alexei Ermakov, and Wendelin Werner, \emph{On some annihilating
  and coalescing systems}, Journal of Statistical Physics \textbf{91} (1998),
  no.~5-6, 845--870.

\bibitem[TW83]{anti-particle}
Doug Toussaint and Frank Wilczek, \emph{Particle-antiparticle annihilation in
  diffusive motion}, Journal of Chemical Physics \textbf{78} (1983), no.~5,
  2642--2647 (English (US)).

\end{thebibliography}
	
\end{document}